\documentclass{conm-p-l}
\usepackage[cp1251]{inputenc}
\usepackage[russian,english]{babel}
\usepackage{amsmath,amsthm,amssymb,graphicx,color}
\usepackage[colorlinks, pdfauthor={Yu. Brezhnev},,pdfwindowui=false,
            bookmarksopen=false, 
            bookmarks=false]{hyperref}
\usepackage[X2,T2A]{fontenc}
\newcommand{\е}{{\fontencoding{X2}\selectfont\cyryat}}

\def~{\unskip\nobreak\hspace{0.27778em}\ignorespaces}
\def\,{\ifmmode\mskip+1.5mu\else\kern+.08333em\fi\relax}
\def\!{\ifmmode\mskip-1.5mu\else\kern-.08333em\fi\relax}
\newdimen{\FontSize} \delimiterfactor=990

\def\ds{\displaystyle}  \def\sss{\scriptscriptstyle}
\def\ie{i.\,\,e.}  \def\eg{e.\,\,g.}
\def\re{{\,\mathrm{e}}} \def\ri{{\mathrm{i}}}

\def\wpp{\wp\hbox{\smaller[1]$'$}}

\def\omp{\omega\hbox{\smaller[1]$'$}}

\def\DEF{\mathrel{\vcenter{\hbox{$:$}}{=}}}
\def\FED{\mathrel{{=}\vcenter{\hbox{$:$}}}}
\def\Dtheta{{\ds\theta'\nobreak{}_{\mskip-7.5mu1}}\mskip-1.5mu\relax}

\def\smin{\mathrel{\vcenter{\hbox{\scalebox{0.7}[1]{$\scriptstyle-$}}}}}
\def\hence{=\kern-0.5em\Rightarrow}
\def\?{\textrm{\protect\footnotesize$\RED\mathchar"446$}}
\def\vphi{\mskip0mu\raisebox{-0.54ex}{\scalebox{0.6}[1]{\hbox{$-$}}}
 \mkern-7.76mu\raise0.52ex\hbox{$\varphi$}}
\def\vpsi{\mskip0.4mu\raisebox{-0.54ex}{\scalebox{0.6}[1]{\hbox{$-$}}}
 \mkern-8.82mu\raise0.52ex\hbox{$\psi$}{}}

\def\russian{\selectlanguage{russian}}
\def\english{\selectlanguage{english}}
\def\END{\end{document}}

\def\Mfrac#1#2{\def\FRAC{\hbox{\smaller[1]$\ds\frac{#1}{#2}$}}
\ifdim\FontSize=11pt\raise0.051ex\FRAC\else\raise0.059ex\FRAC\fi}
\def\mfrac#1#2{\def\FRAC{\hbox{\smaller[3]$\ds\frac{#1}{#2}$}}
\ifdim\FontSize=11pt\raise0.13ex\FRAC\else\raise0.174ex\FRAC\fi}
\newcommand{\AB}[2]{\nobreak\mskip-1mu
\raise0.12ex\hbox{\smaller[2]\renewcommand{\arraystretch}{0.42}%
\setlength{\arraycolsep}{0pt}\raise0.16ex\hbox{$\big[$}%
$\begin{array}{c}\scriptstyle#1\\ \scriptstyle#2\end{array}$%
\raise0.16ex\hbox{$\big]$}}\mskip-1mu\relax}

\newcommand{\mbig}[2][4]{\vcenter{\hbox{%
\ifcase#1$#2$\or\small$\big#2$\or$\big#2$\or\large$\big#2$%
\or\footnotesize$\Big#2$\or\small$\Big#2$\or$\Big#2$\or\large$\Big#2$%
\or$\bigg#2$\or\large$\bigg#2$\or$\Bigg#2$\or\large$\Bigg#2$%
\or\Large$\Bigg#2$\else\huge$\bigg#2$\fi}}\relax}

\newdimen{\upp}\newdimen{\D}\newdimen{\DD} 
\newcommand{\pow}[4][\displaystyle]{
\settowidth{\upp}{\hbox{$#1 {}^{#4}$}}%
\settowidth{\D}{\hbox{$#1{#2}$}}%
\settowidth{\DD}{\hbox{$#1{#2}_{#3}$}}%
\addtolength{\DD}{-\D} {#1{#2}_{#3}\kern-\DD {}^{#4}}
\addtolength{\DD}{-\upp}\ifdim\DD>0pt \kern+\DD \fi\relax}

\makeatletter 
\providecommand{\larger}[1][1]{
\ifmmode\relax\else\@tempcnta\ifx\@currsize\normalsize 4\else
 \ifx\@currsize\small 3\else\ifx\@currsize\footnotesize 2\else
 \ifx\@currsize\large 5\else\ifx\@currsize\Large 6\else
 \ifx\@currsize\LARGE 7\else\ifx\@currsize\scriptsize 1\else
 \ifx\@currsize\tiny  0\else\ifx\@currsize\huge 8\else
 \ifx\@currsize\Huge  9\else 4 \fi\fi\fi\fi\fi\fi\fi\fi\fi\fi
\advance\@tempcnta#1\relax\ifnum\@tempcnta<\z@ \@tempcnta\z@ \fi
\ifcase\@tempcnta\tiny\or\scriptsize\or\footnotesize\or\small
 \or\normalsize\or\large\or\Large\or\LARGE\or\huge\else\Huge\fi\fi}

\providecommand{\smaller}[1][1]{\larger[-#1]} \makeatother

\makeatletter \def\@dotsep{1} \makeatother 

\def\RED{\color{red}} \def\BLUE{\color{blue}} 

\newcommand{\Width}[2][\displaystyle]
{\settowidth\D{\hbox{$#1#2$}}\texttt{\tiny\BLUE$($\the\D$)$}}
\newcommand{\Height}[2][\displaystyle]
{\settoheight\D{\hbox{$#1#2$}}\texttt{\tiny\BLUE$($\the\D$)$}}
\newcommand{\Depth}[2][\displaystyle]
{\settodepth\D{\hbox{$#1#2$}}\texttt{\tiny\BLUE$($\the\D$)$}}

\theoremstyle{plain}
\newtheorem{theorem}{Theorem}[section]
\newtheorem{lemma}[theorem]{Lemma}
\newtheorem{corollary}[theorem]{Corollary}
\newtheorem{proposition}[theorem]{Proposition}

\theoremstyle{definition}
\newtheorem{definition}[theorem]{Definition}
\newtheorem{example}{\sf Example}

\theoremstyle{remark}
\newtheorem{remark}{Remark}

\allowdisplaybreaks[4]

\def\gk{\gamma_k^{}}
\def\N{\mathcal{N}}
\def\Nu{\mathcal{N}(\protect\dpotu)}
\def\NTheta{\mathcal{N}(\Theta)}
\def\II{{\boldsymbol{\mathrm{I\kern-0.15em I}}}}
\def\III{{\boldsymbol{\mathrm{I\kern-0.18em I\kern-0.18em I}}}}
\def\p{\mathcal{P}}
\def\G{{\boldsymbol{\mathfrak{G}}}}
\def\R{{\boldsymbol{\mathcal{R}}}}
\def\bL{{\widehat{\boldsymbol{L}}}}
\def\bA{{\,\,\widehat{\!\!\boldsymbol{A}}}}
\def\potU{\hbox{\smaller[2]$U$}}
\def\pot{\hbox{\smaller[2]%
\raise0.07em\hbox{$\{$}$U$\raise0.07em\hbox{$\}$}}}
\def\dpotU{\hbox{\smaller[2]%
\raise0.07em\hbox{$ [$}$U$\raise0.07em\hbox{$]$}}}
\def\dpotu{\hbox{\smaller[2]%
\raise0.07em\hbox{$[$}{\larger[2]$u$}\raise0.07em\hbox{$]$}}}
\def\dpotv{\hbox{\smaller[2]%
\raise0.07em\hbox{$[$}{\larger[2]$v$}\raise0.07em\hbox{$]$}}}

\author[Yu.~Brezhnev]{Yurii V.~Brezhnev}

\title[Spectral/quadrature duality of finite-gap potentials]
{Spectral/quadrature duality:\\
Picard--Vessiot theory and finite-gap potentials}


\date{\today}
\email{brezhnev@mail.ru}
\thanks{Research supported by the Federal Targeted Program
under contract 02.740.11.0238 and partially by \textsc{Royal
Society/NATO} and RFBR grant (00--01--00782).}

\begin{document}

\maketitle 

\begin{abstract}
In the framework of differential Galois theory we treat the classical
spectral problem $\Psi''-u(x)\Psi=\lambda\Psi$ and its finite-gap
potentials as exactly solvable in quadratures by Picard--Vessiot
without involving  special functions; the ideology goes back to the
1919 works by J.~Drach. We show that duality between spectral and
quadrature approaches is realized through the Weierstrass permutation
theorem for a logarithmic Abelian integral. From this standpoint we
inspect known facts and obtain new ones: an important formula for  the
$\Psi$-function and $\Theta$-function extensions of Picard--Vessiot
fields. In particular, extensions by Jacobi's $\theta$-functions lead
to the (quadrature) algebraically integrable equations for the
$\theta$-functions themselves.
\end{abstract}

\tableofcontents
\newpage
\section{Introduction}

Initially  the method of finite-gap integration was developed in works
\cite{19,17,21, 55,37,7,8} as a periodic generalization of the
celebrated inverse scattering transform method (the soliton theory
\cite{newell}). In the very first papers on this topic
\cite{19,21,4,37,7} it has become clear that  analogs and
generalizations of the soliton potentials are to be smooth and real
functions $u(x)$ in the spectral problem defined by the Schr\"odinger
equation
\begin{equation}\label{1}
\Psi_\mathit{\!xx}-u(x)\,\Psi=\lambda\,\Psi\,,
\end{equation}
if the continuous spectrum of the problem consists of finitely many
`forbidden gaps'\footnote{If  parameter $\lambda$ belongs to such a
lacunae (it is a line segment on the real axe $\lambda$), then the
$\Psi$-function growths unboundedly as a function of $x$.} (lacunae,
bands, zones, intervals \cite{37,55}). This explains the widely used
terminology `finite-gap', abbreviated further as FG. Ensuing
development of the theory  went far beyond equation \eqref{1} and took
the algebro-geometric characterization \cite{29,24,25,5}. Riemann
surfaces and their theta-functions have become the main subject of
study \cite{4,5,29}. Appearance of these nontrivial objects is dictated
by the very nontrivial dependence of solution to Eq.~\eqref{1} upon
parameter $\lambda$ and the search for this dependence is a starting
and prime subject of the spectral \cite{37,17,21} and algebro-geometric
($\Theta$-function) \cite{24,5} approaches. These theories are referred
frequently to as the $\Theta$-function integration. Presently, one can
say that the intense study of equation \eqref{1} over the last decades
led to the fact that its FG-theory has been developed almost
exhaustively. In this connection, it is, perhaps, not without interest
to consider one more view on integration of the problem \eqref{1}.

\subsection{Motivation\label{motiv}}
In the early 1980s some authors revealed the two old papers by J.~Drach
\cite{49_1,49_2} wherein equation \eqref{1} was integrated `directly'
and main results of the theory  were presented in extremely condensed
form. Although these works had subsequently received some mention in
the literature \cite{23,9,newell}  with a special emphasis to the
FG-theory (\cite[pp.~84--85]{38}; written by Matveev), some surprising
facet is the fact that more detailed exposition of Drach's ideology has
not been presented in the modern literature hitherto. The need for such
exposition is apparent when taken into account that the
works\footnote{After these works Drach had not longer returned to
integration of linear ODEs.} \cite{49_1,49_2} themselves contain no any
explanations or proofs. In this connection, it is pertinent to make up
for this gap and sketch an appropriate theory.

The original approach by Drach is to integrate \eqref{1} as an ODE.
Indeed, equation \eqref{1} is primarily a differential equation in
variable $x$ even though we consider it in the algebro-geometric
\cite{24,32} or spectral context \cite{21,37} in which  the $\Psi$ is
viewed, primarily, as function of $\lambda$. Anyway, in so far as the
$\Psi(x;\lambda)$ is a function of two variables, complete theory must
explain this duality and therefore provide conversion between  `$x$-'
and `$\lambda$-formulae'. On the other hand,  integration of linear
ODEs is the subject matter of the old and well developed differential
extension of the algebraic Galois theory which is variously known as
the Picard--Vessiot theory and sometimes as the Lie--Kolchin theory.
The main references in this topic are monographs
\cite{singerput,kapl,berkovich} and classical works
\cite{singer,kolchin}. Strange though it may seem, the explicit
discussion of a direct linkage between this theory and the modern
aspects of integrable models associated to  Eq.~\eqref{1} appeared
comparatively recently \cite{ispanec,A2}.

Correlation between the Picard--Vessiot theory and $\Theta$-function
methods brings up the following question: what is the relation between
these two techniques upon applying them to the linear spectral
problems, say, \eqref{1}? The formal answer (commonly accepted (?))
might be the following: the $\Theta$-series (see next section for
definition) is a \textit{special function} solving \eqref{1}. Indeed,
an integration theory by Picard--Vessiot begins with the precise
definition of a used function class \cite{singerput,kolchin}. By this
is meant that, in the strict sense, without definition of the
integrability domain, any scheme based on use of the formally
postulated $\Theta$-series should indeed be considered as integration
in terms of special functions. We shall show that this is not the case.
For example, the $\Theta$-function integration of Hamiltonian
finite-dimensional nonlinear dynamical systems $\dot q=V(q)$ is know to
be a manifestation of their algebraically invariant Liouvillian
integrability. On the other hand, integrability of such systems is very
well known to be related \cite{dik4} to their representability through
certain linear equations (Lax pairs):
$$
\big\{\boldsymbol{L}(q)\Psi=\lambda\,\Psi\,,\quad \dot\Psi=
\boldsymbol{A}(q)\,\Psi\big\}
\qquad\Rightarrow\qquad
\dot{\boldsymbol{L}}=\big[\boldsymbol{A},\boldsymbol{L}\big]
\quad\Leftrightarrow\quad \dot q=V(q)\,.
$$
The `nonlinear Hamiltonian' integrability is certain to entail the
`linear Picard--Vessiot' one, if only because the logarithmic
derivative $\dot\Psi/\Psi$ is a rational function of dynamical
variables. It should be noted here that the development of this
ideology can serve as a basis for independent concept of integration of
spectral problems at all. However, we do not touch here on such
Hamiltonian systems and algebro-geometric (FG) integration of partial
differential equations (PDEs). We focus only on a linear problem as
such, so  our main intention with this work is to show that the scheme
of FG-integration of the linear spectral equation \eqref{1} should be
separated into the two parts:
\begin{enumerate}
\item The \textit{invariant property} of equation \eqref{1} to be
integrable, \ie,  Lie--Kolchin's solvability of corresponding
differential Galois group \cite{kolchin,singerput}.
\item \textit{Representation} of differential fields and solutions in
terms of those functions which of inevitably appear in the theory.
These are the $\Theta$-series. In some particular cases the series
themselves satisfy the \textit{algebraically} integrable
\textit{Hamiltonian} ODEs.
\end{enumerate}

By the invariance, here and in the subsequent discussion, we shall
informally mean an independence of representations by theta-functions.
Notwithstanding the fact that representation of solutions requires
introducing the highly nontrivial transcendental $\Theta$-objects, the
integrability mechanism itself is very simple. It coincides in effect
with an elementary solvability in closed
form\footnote{\label{kov}J.~Kovacic, in his famous work \cite{kov} on
p.~4, notes: `by a ``closed-form'' solution we mean, roughly, one that
can be written down by a first-year calculus student'. As we shall see,
this `definition' is completely compatible with the transcendental
$\Theta$-function characterization of the FG-integration of
Eq.~\eqref{1}. Rephrasing, there is a closed form solution
(Theorem~\ref{T3}) that can be verified by a direct substitution into
\eqref{1} followed by use of the first-year student calculus: algebra
and differentiation.} and thereby trivializes understanding of the
major portion of the FG-theory. In other words, our main purpose is to
bring the Picard--Vessiot aspects---fields, their extensions,
differential Galois group, quadrature solvability, etc---into the
foreground and, subsequently, to get representations for them in the
FG-terms---spectral curves, variables $\gk$, $\Theta$'s, etc. The
latter objects, to the best of our knowledge, have not received mention
in the contemporary works on the Picard--Vessiot integration of linear
ODEs. See, for example, work \cite[first sentence in
\S\,3(b)]{previato}, monographs \cite{singerput, ispanec, berkovich},
and quite voluminous references therein. Partially, some fragments of
the theory, in a context of the elliptic Lam\'e potentials
$u=A\,\wp(x)$, can be found in book \cite{ispanec} and  work
\cite{etingof}. We note also that FG-potentials are the Abelian
functions \cite{4} and Abelian extensions of differential fields were
already briefly considered by Kolchin himself \cite{kolchin2}. There is
no escape from the mentioning nice applications of the Picard--Vessiot
theory of Eq.~\eqref{1} to the supersymmetric quantum mechanics. They
appeared comparatively recently and this theme is the subject matter of
recent works \cite{A1,A2}.


\subsection{Outline of the work}
Section~\ref{2} contains the background material: a sketch of the
classical Burchnall--Chaundy theory of commuting operators and its
modern formulation in the language of theta-functions.

In Sect.~\ref{3} we briefly recall the needed facts from the
Picard--Vessiot theory, introduce the base differential fields
(Novikov's fields), and motivate their hyperelliptic extensions. Then
we describe a structure of the differential Galois group for
FG-potentials.

Section~\ref{4} is devoted to the quadrature (Drach) characterization
of the FG-integrability and an explanation as to how the known
Weierstrass theorem on an Abelian logarithmic integral performs the
transition and difference between quadrature and spectral mode of
getting formulae for the $\Psi$-function.

In Sect.~\ref{5} we first recall that the theta-function formulae can
be derived from quadratures ones \cite{br3} and then show the necessity
to represent  the previous base objects in terms of theta-functions
and, in particular, to introduce an important object---the
1-dimensional section of the theta-function argument with a free
parameter. Owing to some differential properties of theta-functions the
theory acquires very effective form in those cases when jacobians are
reducible to a product of elliptic curves.

In Sect.~\ref{6} we completely pass to the theta-function
representations and give an appropriate formulation to the
FG-Picard--Vessiot integrability of Eq.~\eqref{1}. This provides a nice
analogy with solubilities in the simplest integrability domains like
field $\mathbb{C}$.

Section~\ref{7} explains how the theta-function reformulation of
Picard--Vessiot integrability transforms into the differential
closedness of the theta-functions themselves.  This also gives a new
treatment to the spectral parameter and a relationship of this
treatment to the closedness and linearity of some of defining
equations. We expound results at greater length for the cases when
multi-dimensional $\Theta$ reduces to the 1-dimensional Jacobian
$\theta$'s. By way of illustration we exhibit a simplest $g=2$
non-elliptic potential.

Differential properties of $\theta$-functions described in previous
section allows us to take these as a starting point for definition of
the functions themselves. It turns out that such a view leads again to
a Liouvillian extension but the latter is accompanied by introducing a
meromorphic elliptic integral and brings up some questions about
differential structures of the multi-dimensional $\Theta$. All this is
expounded in Sect.~\ref{8}.

In Sect.~\ref{9} we exhibit some counterexamples fitting no to the
canonical FG-theta-theory but being certainly integrable \`a la
Picard--Vessiot with the solvable Galois groups. One of good examples
is the famous and fundamental Hermitian equation (containing a
parameter) very closely related to the theory of Eq.~\eqref{1}.

Section~\ref{10} contains some conclusive remarks. \vfil

\section{Background\label{2}}
\subsection{Commuting operators}
The standard soliton/FG-solutions of integr\-able equations are known
to be defined through the associated linear PDEs for the auxiliary
$\Psi$-function:
\begin{equation}\label{L}
\bL(\pot;\partial_x)\Psi=\lambda\,\Psi\,,\qquad
\Psi_t=\bA(\pot;\partial_x)\,\Psi\,,
\end{equation}
where $\bL$ and $\bA$ are the ordinary differential (scalar or matrix)
operators with coefficients \pot\ being, in general, some functions of
variables $(x,t)$: $\pot\DEF\bigl\{u_k^{}(x,t)\bigr\}$. The set of
function \pot\ is usually termed as the potential. This is because the
first of Eqs.~\eqref{L} contains no $\partial_t$ and thereby may be
considered as a spectral problem defined by the operator expression
$\bL$. Moreover, the classical (spectral \cite{21}) property of the
potential \pot\ to be finite-gap  is determined by the spectrum of this
eigenvalue problem and does not depend on $t$. On the other hand, the
classical Burchnall--Chaundy--Baker (algebro-geometric) formulation
\cite{31,32} uses one more spectral problem instead of second equation
in \eqref{L}:
\begin{equation}\label{A}
\bL(\pot;\partial_x)\Psi=\lambda\,\Psi\,,\qquad
\bA(\pot;\partial_x)\Psi=\mu\,\Psi\,,
\end{equation}
In both the formulations the nontrivial theory appears if equations
\eqref{L} or \eqref{A} are compatible; there exists a common solution
$\Psi$ and the potential is subjected to certain conditions. These are
the well-known commutativities of operators
$\big[\bL,\,\partial_t-\bA\big]=0$ for \eqref{L} and
$\big[\bL,\,\bA\big]=0$ for \eqref{A}.

\subsubsection{Why one should pass from \eqref{L} to \eqref{A}?}
There is a simple explanation as to this question. Let the operator
$\bA$ be determined by an hierarchy of some integrable PDEs
\begin{equation}\label{K}
\potU_{\!t}=\mathsf{K}(\dpotU)\,,
\end{equation}
where, as usual in a formal differential calculus \cite{28} and in the
sequel, the symbol $\dpotU$ is used to denote the finite set of
derivatives $\{\potU$, $\potU_{\!x}$, $\potU_{\!\mathit{xx}},\ldots\}$.
Such hierarchies have been well tabulated in the literature
\cite{dik4,GH}. Regard Eqs.~\eqref{L} from the viewpoint of their
explicit integration (in some sense of the word). In a straightforward
statement this problem is impossible to solve because $t$ is a hidden
variable in the first of Eqs.~\eqref{L}. In fact, this variable may be
thought of as an additional spectral one\footnote{This is so indeed
because any additional parameter in coefficients of
$\bL(\{U\};\partial_x)$ may be formally viewed as a spectral variable
defining a spectral operator pencil.} and the $t$-dependence of the
potential \pot\ is unknown/undetermined. Complexity of the question is
not reduced until Eqs.~\eqref{L} remain \textit{partial} differential
equations. Indeed, these equations do not have a general solution
expressible in terms of any known functions. This is because
Eqs.~\eqref{K}, integrable as they are, are not solvable in general.
The only way, in order to solve the question, is to consider some
particular situations when  PDEs \eqref{K} admit transformations into
some ODEs. Clearly, such a possibility is related to the self-similar
reductions of Eqs.~\eqref{K} and the most simple case is of course the
reduction to the stationary variable $z=x-c\,t$. Assuming now the
dependence $\potU=\potU(x-c\,t)$, we get, instead of \eqref{L},
$$
\bL(\potU(z);\partial_z)\Psi=\lambda\,\Psi\,,\qquad
\Psi_t=\bA(\potU(z);\partial_z)\,\Psi\,.
$$
The time $t$ disappears in  the first of these equations and therefore
there exists a solution in form of separability of variables:
\hbox{$\Psi=T(t)\cdot\psi(z)$}. Substituting this ansatz into the last
equations, we get immediately a separability parameter\footnote{An
independent treatment of the second eigenvalue of the second operator.}
$\mu$ and an exponential dependence $T=\exp(\mu\,t)$. Equations
\eqref{L} thus acquire  form \eqref{A}:
$$
\bL(\pot;\partial_z)\,\psi=\lambda\,\psi\,,\qquad
\big(\bA(\pot;\partial_z)+c\,\partial_z\big)\,\psi=\mu\,\psi\,,
$$
where $\potU=\potU(z)$. Compatibility conditions of these equations
$\big[\bL,\,\bA+c\,\partial_z\big]=0$ generate the stationary
Lax--Novikov equations  $F(\dpotU)=0$ \cite{GH} and, incidentally,
different kind reductions may lead to other kind equations. For
example,  the general form of Painlev\'e equations and their
`$(\bL,\bA)$-pairs' can be obtained by this way \cite{nijhoff}.

\subsubsection{Algebraic curve and integrals of Lax--Novikov equations}
Let $n$, $m$ be the orders of the operators $\bL$, $\bA$ respectively.
Assuming that Eqs.~\eqref{A} are compatible, we conclude, by
elimination of the $\Psi$ from \eqref{A}, that parameters $\lambda$ and
$\mu$ are related by a \textit{polynomial} dependence
\hbox{$W(\stackrel{n}{\hbox{\smaller[2]$\mathstrut$}\mu},\,
\stackrel{m}{\hbox{\smaller[2]$\mathstrut$}\phantom{\mu}}
\hspace{-0.9em}\lambda)=0$} and commuting operators $\bL$, $\bA$
themselves are also tied by the same dependence: $W(\bA,\,\bL)$ is a
zero operator \cite{31,32}. The identity $W(\bA,\,\bL)=0$  implies that
its coefficients, being the differential functions $E_j$ of \pot, must
be free constants $C_j$ independent of $x$. These yields a set of
(compatible) ODEs $E_j(\dpotU)=C_j$ providing some integrals of motion
$E_j$ for the Lax--Novikov equations mentioned above \cite{dik4}. The
further theory requires that the potential \pot\ be the complex
analytic function of the complex argument $x$.

\subsection{Theta-function formulae}
Solution to Eqs.~\eqref{A} is an $n$-valued function of $\lambda$ (and
an $m$-valued function of $\mu$) and this multi-valuedness is related
to the algebraic equation $W(\mu,\lambda)=0$.  This equation, being
viewed as an algebraic curve over $\mathbb{C}$, defines a compact
Riemann surface $\R$ of a finite genus. It is well known that
multi-dimensional $\Theta$-functions are the universal tool in order to
impart a `single-valued form' to the analytic apparatus on Riemann
surfaces of multi-valued algebraic functions \cite{baker}. Baker
\cite{32}, initiated by work \cite{31}, was the first to transfer this
$\lambda$-multi-valuedness of the $\Psi$ into a single-valued function
of a point $\p\in\R$ and to construct the function itself through
Riemann's $\Theta$-functions; the very first sentence of the work
\cite{32} indicates this.  Akhiezer \cite{15} arrived at the same
objects when considering the problem \eqref{1} from completely
different---purely spectral---viewpoint. In the 1970s all these
discoveries were substantially developed, generalized
\cite{17,21,4,8,24,25,29}, and acquired their current $\Theta$-function
form. The recent excellent survey by Matveev \cite{matv2} is, perhaps,
the best work both on the history of the question and background
material. Most general and modern treatment of these results goes back
to works by Krichever and reads as follows.

The spectral variable is thought of as the meromorphic function
$\lambda=\lambda(\p)$ in the sense that the complex number $\lambda$ in
problem \eqref{L} is replaced with an abstract coordinate on $\R$: the
point $\p$. The variable $x$ is viewed as a parameter now and the
function $\Psi$, as function on $\R$, is the function $\Psi(\p)$ of an
exponential type \cite{15,32} with essential singularities of some
prescribed form \cite{4,24}. The formula realization of this result is
known presently as a concept of the Baker--Akhiezer (BA) function
\cite{38}. The operator $\bL$ as above and its coefficients (the
potential $\pot$) is called the \textit{finite-gap} or
\textit{algebro-geometric}. Of course, we could equally well say the
same about $\mu=\mu(\p)$ for operator $\bA$ in problem \eqref{A}.

The structure of  solutions is universal \cite{32,24}. For example, the
scalar problem
\begin{equation}\label{Lg}
\bL\Psi\DEF\frac{d^n}{dx^n}\,\Psi+u_2^{}(x)\,
\frac{d^{n\smin2}}{dx^{n\smin2}}\,\Psi+\cdots+ u_n(x)\,\Psi=
\lambda\,\Psi\,,
\end{equation}
and therefore equation \eqref{1}, in the class of FG-potentials, has a
solution which is given, under some normalization, by a general
formula:
\begin{equation}\label{theta}
\Psi\big(x;\,\lambda(\p)\big)=
\frac{\Theta\big(x\,\boldsymbol{U}+
\boldsymbol{D}+\boldsymbol{\mathfrak{U}}(\p)\big)}
{\Theta(x\,\boldsymbol{U}+\boldsymbol{D})}\,
\re^{\II(\p)\,x}\,.
\end{equation}
Distinctions between different FG-potentials $\{u_k(x)\}$ are only in
the changes of the associated algebraic curve $W(\mu,\lambda)=0$ and
the curve itself (its $\R$ with a canonical homology base
$(\boldsymbol{\mathfrak{a}},\boldsymbol{\mathfrak{b}})$) determines all
the quantities appearing in \eqref{theta}. Namely:
$\Theta(\boldsymbol{z})$ is the canonical $g$-dimensional theta-series
\begin{equation}\label{series}
\Theta(\boldsymbol{z})\DEF\Theta(\boldsymbol{z}|\boldsymbol{\Pi})=
\ds\sum\limits_{\hbox{\tiny$\boldsymbol{N}$}
\hbox{\footnotesize\raise0.01em\hbox{$\in$}}
\hbox{\tiny\,$\mathbb{Z}$}^g_{\mathstrut}}
{\ds\mathstrut}\!\!\!\re^{\pi \ri\, \langle
\boldsymbol{\Pi}\boldsymbol{N}\!,\,\boldsymbol{N} \rangle+ 2\pi
\ri\,\langle \boldsymbol{N}\!, \,\boldsymbol{z}\rangle}_{\mathstrut}
\end{equation}
of $g$ arguments $\boldsymbol{z}=(z_1,\ldots,z_g)$, built by an
$\boldsymbol{\mathfrak{a}}$-periods $\boldsymbol\Pi$-matrix of
normalized holomorphic Abelian integrals
$\boldsymbol{\mathfrak{U}}(\p)= \big(\mathfrak{U}_1(\p),
\ldots,\mathfrak{U}_g(\p)\big)$ on $\R$; symbols like
$\langle\boldsymbol{N}\!,\,\boldsymbol{z}\rangle$ denote Euclidian
scalar product $\langle\boldsymbol{N}\!,\,\boldsymbol{z}\rangle\DEF
N_jz_j$; $\p$ is a free point on $\R$; $\II(\p)$ is a normalized
Abelian integral of the second kind with only pole of the first order
at a point $\p_{\!\!\infty}$ at which
$\lambda(\p_{\!\!\infty})=\infty$; the vector $\boldsymbol{U}$ times
$2\,\pi\,\ri$ is a vector of $\boldsymbol{\mathfrak{b}}$-periods of
this integral; $\boldsymbol{D}$ is an arbitrary constant $g$-vector.
All this terminology is exhaustively expounded in the numerous
literature (see, \eg, \cite{baker,38,clebsch,5,29,GH,4}) and the
algebraic dependence $W(\mu,\lambda)=0$  is referred frequently to as
the spectral curve. As for the problem \eqref{1}, this curve
constitutes a hyperelliptic equation of the form
\begin{equation}\label{hyper}
\mu^2=(\lambda-E_1)\cdots(\lambda-E_{2g+1})
\end{equation}
and all the FG-potentials are given by the famous Its--Matveev formula
\cite{4,17}
\begin{equation}\label{its}
u(x)=-2\,\frac{d^2}{dx^2}\ln\Theta(x\,\boldsymbol{U}+\boldsymbol{D})
+\mathrm{const}\,.
\end{equation}

\begin{example}
Most simple and popular  example is a 1-gap potential. It is unique and
is determined by the Weierstrass elliptic curve
\begin{equation}\label{torus}
\mu^2=4\,(\lambda-e)(\lambda-e\hbox{\smaller[2]$'$})
(\lambda-e\hbox{\smaller[2]$''$})\,;
\end{equation}
we denote its modulus as $\hbox{\footnotesize$\Pi$}=\omp\!\!/\omega$,
where $\omega$, $\omp$ are   Weierstrassian half-periods normalized by
the condition $\boldsymbol{\Im}\,\hbox{\footnotesize$\Pi$}>0$
\cite{ell}. Since this case is a 1-dimensional one, we replace
$\p\mapsto\mathfrak{u}$ and put\footnote{Jacobian of an elliptic curve
is isomorphic to the curve itself.}
$\boldsymbol{\mathfrak{U}}(\p)=\mathfrak{u}$. Therefore, normalizing
integral $\II$, we have
$$
\II(\mathfrak{u})=\zeta(2\,\omega\,\mathfrak{u})-
2\,\eta\,\mathfrak{u}\quad\Rightarrow\;\;
\left\{
\begin{aligned}
\II(\mathfrak{u}+1)&=\II(\mathfrak{u})\\
\II(\mathfrak{u}+\hbox{\footnotesize$\Pi$})&=\II(\mathfrak{u})+
\frac{\pi}{\ri\,\omega}\quad \Rightarrow\quad
\boldsymbol{U}=-\frac{1}{2\,\omega}\,,
\end{aligned}
\right.
$$
where $\zeta(z)\DEF\zeta(z|\omega,\omp)$ and
$\eta\DEF\eta(\omega,\omp)$ are the standard objects accompanying the
theory of Weierstrassian function $\wp(z)\DEF\wp(z|\omega,\omp)$
\cite{ell,WW}.  Putting for simplicity $\boldsymbol{D}=0$, formulae
\eqref{theta} and \eqref{its} become
$$
u(x)=-2\,\frac{d^2}{dx^2}\ln\Theta\!\!\left(\Mfrac{x}{2\,\omega}
\mbig|
\hbox{\footnotesize$\Pi$}\right)-2\,\frac{\eta}{\omega}\,,
\qquad \Psi(x;\lambda)= \frac{\Theta
\mbig[5](
\mfrac{x}{2\,\omega}-\mathfrak{u}\mbig| \hbox{\footnotesize$\Pi$}
\mbig[5])}
{\Theta\mbig[5](\mfrac{x}{2\,\omega}\mbig| \hbox{\footnotesize$\Pi$}
\mbig[5])}\,
\re^{\left\{\zeta(2\omega\mathfrak{u})-
2\eta\mathfrak{u}\right\}\,x}_{\mathstrut}
$$
and $\lambda=\wp(2\,\omega\,\mathfrak{u})$ (this is the formula
$\lambda=\lambda(\p)$ above). This potential is a precise equivalent of
the classical Lam\'e form $u(x)=2\,\wp(x-\omega-\omp)$ and
1-dimensional $\Theta(z|\hbox{\footnotesize$\Pi$})$-series here
coincides exactly with the Jacobi function
$\theta_3(z|\hbox{\footnotesize$\Pi$})$ defined by the standard formula
\eqref{jacobi}.
\end{example}

All the constructions mentioned above---spectral, algebro-geometric,
$\Theta$-func\-tion, and their varieties---are completely equivalent
\cite{25}, which is why we shall refer to these approaches merely as
spectral for short.

\section{Integrability of equation \protect\eqref{1}
by Picard--Vessiot\label{3}}

\subsection{The function $R$} The theory of equation \eqref{1} is
closely related to the fundamental linear differential equation
\begin{equation}\label{hermite}
R_\mathit{xxx}-4\,(u+\lambda)R_x-2\,u_xR=0
\end{equation}
determining  function $R=R(x;\lambda)$. Integrating this equation and
denoting an integration constant as $\mu$, we get
\begin{equation}\label{muR}
\mu^2=-\frac12\,R\,R_\mathit{xx}+\frac14\,
\pow{R}{x}{2}+(u+\lambda)R^2
\end{equation}
(see work \cite{br3} for a successive derivation of these formulae with
use of Lie's symmetries approach.) Then the FG-solutions to the
$\Psi$-function for Eq.~\eqref{1}, in general position $\mu\ne0$, are
given by the well-known formula
\begin{align}
\notag \Psi^{\pm}(x;\,\lambda)&=\sqrt{R(x;\lambda)}\,
\exp\!\!\!\int\limits^{\;x}\!\!\!
\frac{\pm\mu\,dx}{R(x;\lambda)}\\[0.3em]
&=\exp\!\!\int\limits^{\;x}\!\!\frac{R_x(x;\lambda)\pm2\,\mu}
{2\,R(x;\lambda)}\,dx\,.\label{pm}
\end{align}
Formulae \eqref{hermite}--\eqref{pm} and their variations have been
repeatedly appeared  in the literature \cite{8,29,74,38,GH}. Their
precise meaning, however, lies in the fact that the availability of
formula \eqref{pm} itself does not mean any integrability \cite{br3}.
This is just an ansatz for Eq.~\eqref{1} and its solution  should be
written down in terms of indefinite integrals; otherwise all the
procedure would reduce to re-notations.

\subsection{General  formula for the $\Psi$-function}
In the language of commuting Burchnall--Chaundy
operators \cite{31}  an explicit formula for the $\Psi$-function
(including \eqref{pm}) results from the sequential elimination of
derivatives $\Psi^{(k)}$ from the pair of differential equations
\eqref{A} down to the formula
\begin{equation}\label{psi1}
\Psi_{\!x}=G(\dpotU;\lambda,\mu)\,\Psi\,\quad\Rightarrow\quad
\Psi=\exp\!\!\int\limits^{\;x}\!\!\!G(\dpotU;\lambda,\mu)\,dx\,.
\end{equation}
This simple recipe of getting the $\Psi$-function, perhaps, has no
received mention in the modern literature \cite{ustinov}; it is,
however, implicitly exploited in monograph \cite{GH}. Elimination of
the last derivative $\Psi_{\!x}$ leads to equation of the curve
$W(\mu,\lambda)=0$. The algebraic dependence $W(\bA,\bL)=0$ serves, in
some of works, as the basis for a formal definition of the algebraic
integrability \cite{25}.

\subsection{Novikov's equations and differential field}
Motivated by the desire to define a differential field over which the
integration is performed, we need to know the differential structure of
function $R(x;\lambda)$. It is well known that this function is a
series in $\lambda$ with coefficients being differential polynomials in
$u(x)$. Computational formulae for these polynomials have been detailed
in the famous work \cite[formulae (8)]{28}. Finite order differential
conditions on the potential appear if equation \eqref{hermite} has a
solution being a polynomial in $\lambda$ \cite{br3}. As in the
FG-theory, the equation \eqref{hermite} is also known in differential
Galois theory as the second symmetric power of operator \eqref{1}
\cite[\S\,4.3.4]{singerput}, \cite[p.~671]{singer1}.

\begin{definition}
The potential $u(x)$ is said to be an FG-potential if equation
\eqref{hermite} has a solution $R(x;\lambda)$ being a polynomial in
$\lambda$. No restrictions on coefficients of the curve \eqref{muR}
$\Leftrightarrow$ \eqref{hyper} have been imposed.
\end{definition}

This definition is not a standard one but it is of course equivalent to
the spectral \cite[Theorem~1]{21}, algebraic \cite[Ch.~3]{28}, or
quadrature \cite{49_2,br3} definitions. The only exception is a formal
$\Theta$-function (algebro-geometric) setting because it does not use
spectra, resolvent $R$, or quadratures. An important point here is the
fact that the function domain for the potential is not defined as usual
in Galois theory but calculated. This calculation is a problem of
nonlinear integration, that is integration of ODEs for the potential
$u(x)$. These are the famous Novikov equations \cite{19}. We shall call
this base field the \textit{Novikov differential field $\Nu$ in
$u(x)$-representation} bearing in mind that $u(x)$ satisfies a Novikov
equation $F(u,u_x,u_{\mathit{xx}},\ldots, \pow{u}{x}{(2g+1)})=0$. The
field $\Nu$ consists of rational functions of $u(x)$ over
$\mathbb{C}(\lambda)$ and its derivatives. Parameter $\lambda$, the $g$
constants $c_k^{}$ coming from the integral Gel'fand--Dickey recurrence
\cite{28}, and branch points $E_k$ of the curve \eqref{muR} belong to a
subfield of constants for $\Nu$. Here is an example of Novikov's
equation under $g=2$:
\begin{equation}\label{N2}
(u_{\mathit{xxxxx}}-10\,u\,u_{\mathit{xxx}}-20\,u_x u_{\mathit{xx}}+
30\,u^2u_x)+c_1^{}(u_{\mathit{xxx}}-6\,u\,u_x)+c_2^{}u_x=0\,.
\end{equation}

Integration constant $\mu$, in an FG-class, is fixed to be dependent on
the parameter in equation, that is $\mu=\mu(\lambda)$, and equation
\eqref{muR} turns into the formula \eqref{hyper}. At the same time, as
soon as $R(x;\lambda)$ becomes a polynomial in $\lambda$ it becomes a
differential polynomial $R(\dpotu;\lambda)\in\Nu$. It should be noted
that one  suffices to have only one solution of Eq.~\eqref{hermite}
belonging to $\Nu$ (see Sect.~\ref{9.1} further below).

\subsection{Picard--Vessiot field, constants, and their hyperelliptic
extension} As usual, the Picard--Vessiot extension
$$
\Nu\mbig[1]\langle\Psi^{\pm}\mbig[1]\rangle
\DEF\Nu\mbig[1](\Psi^+,\Psi^-,
\pow{\Psi}{x}{+},\pow{\Psi}{x}{-}\mbig[1])\,,
$$
\ie, the splitting field, results from attaching to the field $\Nu$
integrals $\Psi^{\pm}$ of equation \eqref{1} and its derivatives
\cite{kolchin,singer,kolchin3,kapl,ispanec}. A simplest kind of
extensions corresponds  to solvable cases of the Galois theory and is
known as the extension of Liouville \cite[p.~33]{singerput},
\cite{ispanec,kapl}.

\begin{definition}
An extension $\widetilde\N$ of the differential field $\N$ is said to
be Liouvillian if there exists a tower of fields
$\N=\N_0\subset\N_1\subset\cdots\subset \N_n=\widetilde\N$ such that
$\N_{k+1}=\N_k(\psi_k^{})$, where $\partial_x\psi_k^{}$ or
$\partial_x\!\ln\!\psi_k^{}$ is an algebraic element over $\N_k$.
\end{definition}

In other words,  Liouvillian extensions are the natural enlargements of
the base field performed by a step-by-step adjunction of solution to
the simplest integrable ODEs: the 1st order linear ODEs
\begin{equation}\label{AB}
\psi_x^{}=A\,\psi+B
\end{equation}
with coefficients $(A,B)$ being algebraic/rational over previous step
field. The standard adjunctions of an algebraic element
$\boldsymbol{a}$, integral $\hbox{\larger$\int$}\!p\,dx$, or
exponential $\exp\!\!\hbox{\larger$\int$}\!p\,dx$, where $p\in\N_k$,
are obtained by putting here $(A,B)=(0,\boldsymbol{a}_x)$,
$(A,B)=(0,p)$, and $(A,B)=(p,0)$ respectively.

Solvability of equation by quadratures is directly connected with a
structure of the group constituting a differential version of the
polynomial Galois group.

\begin{definition}
Differential Galois group $\mathrm{Gal}\big(\N \langle\Psi\rangle\big)$
of a linear ODE defined over $\N$ is a set of linear transformations of
its solutions $\Psi$'s  that preserve all the algebraic $($over
$\N$\/$)$ relations among $\Psi$'s and their derivatives $\Psi^{(n)}$
$($differential automorphisms group\/$)$.
\end{definition}

The Picard--Vessiot extension must have the same set of constants as
$\Nu$ \cite{kolchin}, \cite[p.~411]{kolchin3}, \cite{singerput}. It is
known that for equations of the form \eqref{1} the Wronskian
$\Psi'{\!\!{}_1}\Psi_2-\Psi'{\!\!{}_2}\Psi_1$ is a constant. Taking
\eqref{pm} as one of possible bases for solutions, we get
\begin{equation}\label{w}
\begin{aligned}
\pow{\Psi}{x}{+}\Psi^{-} -\Psi^{+}\pow{\Psi}{x}{-}&=
\sqrt{\pow{R}{x}{2}-2\,R\,R_\mathit{xx}+
4\,(u+\lambda)\,R^2}\\&=2\,\mu(\lambda)\,.
\end{aligned}
\end{equation}
Clearly, it must be a constant  of the Picard--Vessiot field whatever
the potential $u(x)$ may be. Neglecting for the moment other constants,
we obtain that field of constants $\mathbb{C}(\lambda)$ requires
adjunction of the constant $\mu$:
$\mathbb{C}(\lambda)\mapsto\mathbb{C}(\lambda,\mu)$. Character of this
extension is determined by the $\lambda$-dependence of the function
$R$. It may be polynomial, rational, or essentially transcendental. The
case of rational polynomial $R(\lambda)$ is possible only if its poles
do not depend on $x$; for we should otherwise have the
$\lambda$-dependent poles of the $\Psi$-function, which is impossible
by virtue of structure of Eq.~\eqref{1}. Hence the function $R$ must be
an entire function of $\lambda$. The Galois group does depend in
general on parameters of equation but we are interested in the
following cases:
\begin{itemize}
\item \textit{When the Lie--Kolchin integrability structure is the
same for generic $\lambda$?}
\end{itemize}

Therefore two kinds of theories do exist, according as the field
$\mathbb{C}(\lambda,\mu)$ does not, or does, belong to a
\textit{finite} algebraic or \textit{infinite} extension of
$\mathbb{C}(\lambda)$. The latter fields are excessively general
because there are huge varieties of entire transcendental functions
(without any classification) and they do not produce any restrictions
on potential. On the other hand, finite extensions are of fixed
algebraic (necessarily hyperelliptic) character and calibrated by the
only number, namely, by the $\lambda$-degree $g$ of the polynomial
$R(\dpotu;\lambda)$. In this case infinite Gel'fand--Dickey recurrences
terminate and equation \eqref{muR} leads to a finite set of
differential restrictions on $u(x)$ in form of differential polynomials
$E_j(\dpotu)=C_j$. This gives in fact yet another independent
motivation (\`a la Picard--Vessiot) for appearance/availability of a
polynomial solution to Eq.~\eqref{hermite}. For brevity, we shall adopt
however the previous shorter notation for the base field $\Nu$ without
explicit indication of its $\lambda,\mu$-dependence
$\N(\dpotu;\lambda,\mu)$ or dependencies on other field constants
$\N(\dpotu;\lambda,\mu,C_j,\ldots)$, where dots stand for remaining
constants which arise upon complete integration of a Novikov equation.

\subsection{Finite-gap Galois groups}
Below is a characterization of the  Galois  group of equation \eqref{1}
in the class of FG-potentials.  Condition on parameter $\lambda$ of
being an arbitrary quantity is a fundamental requirement meant
throughout the paper.

\begin{theorem}\label{T1}
The Picard--Vessiot extension $\Nu\langle\Psi^{\pm}\rangle$ is a
Liouvillian extension of the transcendence degree equal to $1$.
Associated  group $\mathrm{Gal}\big(\Nu \langle\Psi^{\pm}\rangle\big)$,
under  generic $\lambda\ne E_j$, is connected and isomorphic to the
group $\G=\mbig(
\begin{smallmatrix}\alpha &
0\\0 & \alpha^{\smin1}\end{smallmatrix}\! \mbig)$, where
$\alpha\in\mathbb{C}$. For other values of $\lambda$ it is isomorphic
to group $\G=\mbig(
\begin{smallmatrix}\pm1 &
\alpha\\0 & \pm1\end{smallmatrix}\! \mbig)$.
\end{theorem}

\begin{proof}
For generic $\lambda$'s  integral \eqref{pm} does not belong to $\Nu$.
From \eqref{pm} it follows that this extension is Liouvillian. Take the
canonical basis of solutions \eqref{pm}. Then the three quantities
$\{\Psi^-,\pow{\Psi}{x}{+},\pow{\Psi}{x}{-}\}$ are expressed through
the transcendent being adjoined $\Psi^+$ as follows:
\begin{equation}\label{temp}
\Psi^{-}=\frac{R}{\Psi^{+}}\,,\qquad
\pow{\Psi}{x}{+}=\frac{R_x+2\mu}{2R}\,\Psi^{+}\,, \qquad
\pow{\Psi}{x}{-}=\frac{R_x-2\mu}{2\,\Psi^{+}}\qquad (\mu\ne0)\,.
\end{equation}
Clearly,  in the case of such $\lambda$'s that $\mu=0$ these relations
cease to be valid since $(\Psi^+,\,\Psi^-)$ become linearly dependent
on each other. The relations should be modified and we choose the
following basis $\Psi^\pm$:
\begin{equation}\label{temp2}
(\Psi^-)^2=R\,,\qquad
\pow{\Psi}{x}{-}=\frac{R_x}{2R}\,\Psi^{-}\,,\qquad
\pow{\Psi}{x}{+}\Psi^{-} -\Psi^{+}\pow{\Psi}{x}{-}=1\,.
\end{equation}
In both of these cases we adjoin  one transcendental element
(respectively):
$$
\Psi^+=
\exp\!\!\!\int\limits^{\,\,x}\!\!\!
\frac{R_x+2\,\mu}{2\,R}\,dx
\qquad \text{or}\qquad\Psi^+=\sqrt{\!R\,}\int\limits^{\,\,x}\!\!\!
\frac{dx}{R}\,.
$$
In the latter case the radical $\sqrt{\!R\,}=\Psi^-$ can sometimes be
element of $\Nu$ as an example of the Lam\'e equations shows
\cite[Ch.~23]{WW}, \cite{maier}.

Let us check invariance of relations \eqref{temp}--\eqref{temp2} with
respect to linear transformation of the basis
\begin{equation*}\label{trans}
\left(\begin{matrix}\Psi^+\\\Psi^-\end{matrix}\right)\mapsto
\left(\begin{matrix}\alpha\,\Psi^++\beta\,\Psi^-\\
\gamma\,\Psi^++\delta\,\Psi^-\end{matrix}\right);
\end{equation*}
this determines the Galois group
$\G=\big(\begin{smallmatrix}\alpha & \beta\\
\gamma & \delta\end{smallmatrix}\bigr)$  completely. In case
\eqref{temp} we get the following set of equalities:
$$
\alpha\,\gamma=0\,,\qquad \beta\,\delta=0\,,\qquad
\alpha\,\delta+\beta\,\gamma=1\,,\qquad\beta=0\,.
$$
From this it follows that $\beta=0$, $\gamma=0$, and
$\alpha\,\delta=1$. The number $\alpha$ can not be any algebraic one;
for we should otherwise have a finite Galois group and the algebraic
$\Psi^{\pm}$-solutions to give a contradiction with a single-valuedness
of the general formula \eqref{theta}. In case \eqref{temp2} we derive
that $\delta^2=1$, $\gamma=0$, $\alpha\,\delta=1$, and no conditions on
$\beta$ (except for degenerate cases of curve). Hence in generic case
$\lambda\ne E_j$ the group $\G$ is connected. Its possible forms  are
thus as follows:
$$
\G=\mbig[7](\,
\begin{matrix}\alpha&0\\0&\alpha^{\smin1}\end{matrix}\mbig[7])
\quad \text{or}\quad
\G=\mbig[7](\,
\begin{matrix}\delta&\beta\\0&\delta^{\smin1}\end{matrix}
\mbig[7]),
$$
where $\delta=\pm1$ (compare with cases~3, 4 in Proposition~2.2 of
\cite{ispanec}).
\end{proof}

To summarize briefly, we conclude that independently of the topological
genus of the curve \eqref{hyper}, `finite-gap' groups
$\mathrm{Gal}\big(\Nu\langle \Psi^{\pm}\rangle\big)$ do not depend on
parameter $\lambda$ and cease to be diagonal and connected only for
isolated values of the parameter.

\begin{corollary}\label{C1}
Equations \eqref{1} of the FG-class are factorizable over $\Nu$:
$$
\partial_{\mathit{xx}}-(u+\lambda) =
\left(\partial_x+\frac12\frac{R_x}{R}\pm\frac{\mu}{R}\right)\!\!\!
\left(\partial_x-\frac12\frac{R_x}{R}\mp\frac{\mu}{R}\right).
$$
\end{corollary}

\begin{remark}\label{R2}
We used nowhere any specific form of the polynomial $R$.
Theorem~\ref{T1} is easily restated for arbitrary integrable
$\lambda$-pencils of the 2nd order with the only condition that
$R(x;\lambda)\in\Nu$. Recall that the spectral $\lambda$-pencil is a
generalization of the canonical spectral equation of the form
\eqref{Lg} to  more complex (\eg\ polynomial) dependencies of the
differential expression $\bL$ on the external parameter $\lambda$, that
is $\bL(\pot;\partial_x,\lambda)=0$. An example is the well-known
spectral $\lambda$-pencil of the form
\begin{equation}\label{NLS}
\Psi_\mathit{\!xx}-\frac{u_x}{u}
\,\Psi_{\!x}-\Big(\lambda^2-\frac{u_x}{u}\lambda+u\,v\Big)\,\Psi=0\,;
\end{equation}
it arises when integrating the integrable nonlinear Schr\"odinger
equation \cite{GH,newell,38}. In Sect.~\ref{9} we shall consider other
examples of  the $\lambda$-pencils (see also \cite[Sect.~5(a)]{br3}).
\end{remark}

By virtue of structure of the formula \eqref{psi1} Theorem~\ref{T1} has
a direct generalization.

\begin{theorem}\label{T2}
Let $\N(\dpotU)$ be a Novikov field associated with the pair of
Burch\-nall--Chaundy scalar operators \eqref{A}. Then both of these
equations  are integrable in Liouvillian extensions of the same
transcendence degree. Under generic $\lambda$, $\mu$ the differential
Galois groups of these equations are connected and isomorphic to the
diagonal groups $\G= \mathrm{Diag}(\alpha,\beta,\ldots,\gamma)
\subset\mathrm{GL}(\mathbb{C})$.
\end{theorem}

Novikov's equations are known to be Hamiltonian systems integrable by
Liouville \cite{dik4}. However such a way of their integration is not
necessary since determination of the potential, \ie, construction of
the field $\N(\dpotU)$, is given by formulae following completely from
the `linear' Picard--Vessiot theory. It does not require Hamiltonians.

We finish this section with a digression to one remarkable example. It
is a Matveev 1-positon potential given by the seemingly elementary
formula \cite{posit}
\begin{equation}\label{positon}
u=-2\ln_{\mathit{xx}}\!\!\big\{\!\sin(a\,x+b)-a\,x-c \big\}\,.
\end{equation}
Surprisingly, in spite of its complete fitting into the integration
scheme above, it is not amenable to integration by means of any
classical algorithm in the Picard--Vessiot theory (Kovacic \cite{kov},
Singer \cite{singer}). Indeed, these algorithms are applicable to the
finite algebraic extensions of $\mathbb{C}(x)$, whereas this
$u\in\mathbb{C}(x,\re^{\ri (ax+b)},a,c)$.

\section{Spectral/quadrature duality.
An integration procedure\label{4}}
\subsection{Drach--Dubrovin equations\label{DD}} The quadrature Drach
approach gives a very simple explanation as to why and where the
fundamental polynomial
\begin{equation}\label{factor}
R(\dpotu;\lambda)=\big(\lambda-\gamma_1^{}(x)\big)\cdots
\big(\lambda-\gamma_g^{}(x)\big)
\end{equation}
comes from, what its roots $\gamma_k^{}$ are, and why these are
precisely the quantities that complete the indefinite
quadratures\footnote{An elementary explanation to appearance of these
objects (supplemented with Russian translation of works
\cite{49_1,49_2}) can be found in \cite{br0}.}. In what concerns the
spectral viewpoint, a remarkable result by Dubrovin \cite{7,8} is that
the quantities $\gk$ arise as zeroes of a $\Theta$-function since they
solve the inversion problem of Jacobi \cite{15,4,21}.

\begin{lemma}[Drach \cite{49_2}, Dubrovin \cite{7}]
Functions $\gk(x)$ satisfy the system of ODEs
\begin{equation}\label{dubrovin}
\frac{d\gk}{dx}=2\frac{\sqrt{(\gk-E_1)\cdots
(\gk-E_{2g+1})}}{\displaystyle\prod\limits_{j\ne k}(\gk-
\gamma_j^{})}\,,\qquad j,k=1,\ldots, g
\end{equation}
and  potential is determined by the trace formula \cite{17}
\begin{equation}\label{trace}
u=2\sum\limits_{k=1}^{g}\gk(x)-\sum\limits_{k=1}^{2g+1} E_k\,.
\end{equation}
\end{lemma}

\begin{proof}
One inserts  \eqref{factor} into \eqref{muR} and takes \eqref{hyper}
into account. Collecting the result in degrees $(\lambda-\gk)^n$, one
requires identity under arbitrary $\lambda$. One gets \eqref{dubrovin}
and \eqref{trace}.
\end{proof}

\begin{remark}\label{R3}
We presented such a way of proof because it is applicable to higher
order operators and even spectral $\lambda$-pencils. The reason is that
the derivation of trace formulae is not evident  when generalizing.
Formulae of such a kind cease actually to be the `trace formulae'
because they have no longer Gel'fand's treatment of the operator trace
analogs. They are also not derivable directly from the definition of
$R$-polynomial like \eqref{factor}; illustrative counterexamples can be
found in work \cite{umn}.
\end{remark}

\subsection{Weierstrass theorem and new representation
for the $\Psi$-function} Main content of this section was briefly
announced in \cite{br3} and we shall present here the extensive proofs,
derivations, and precise correlation between spectral and quadrature
approaches.

\begin{theorem}\label{T3}
Let $u(x)$ be an FG-potential  corresponding to the arbitrary curve
\eqref{hyper}. Then solution to equation \eqref{1} is given by the
quadratures
\begin{equation}\label{final}
\Psi^{\pm}(x;\lambda)=\exp
\frac{1}{2}\mbig[10]\{\!\!\!\!
\int\limits^{\,\gamma_1^{}\!(x)}
\!\!\!\!\!\frac{w\pm\mu}{z-\lambda}\,\frac{dz}{w}\,\,\,\,+ \cdots +
\int\limits^{\,\gamma_g^{}\!(x)}
\!\!\!\!\!\frac{w\pm\mu}{z-\lambda}\,\frac{dz}{w}
\mbig[10]\}\;,
\end{equation}
where $w^2=(z-E_1)\cdots(z-E_{2g+1})$ and functions $\gk=\gk(x)$ are
determined through inversion of the set of indefinite integrals
\begin{equation}\label{jac}
\sum\limits_{k=1}^{g}\int\limits^{\;\gamma_k^{}}\!z^{g-1}\,\frac{d
z}{w} = 2\,x+a_g\,,\qquad
\sum\limits_{k=1}^{g}\int\limits^{\;\gamma_k^{}}\!z^n\,\frac{d z}{w}
= a_{n+1}\,,\quad n=0,1,\ldots,g-2\,.
\end{equation}
The FG-potential $u(x)$ is determined by formula \eqref{trace}.
\end{theorem}
\begin{proof}
Let us substitute \eqref{factor} into \eqref{pm} and change the
integration variable $x$ to $z$. We then obtain the following rules
\begin{alignat*}{3}
\frac{R_x}{R}\,dx&=d\ln\prod_k(\lambda-\gk)&&\quad\dashrightarrow\quad
\frac{1}{z-\lambda}\,dz\,,\\[0.3em]
\frac{2\,\mu}{R}\,dx&=
\frac{2\,\mu}{\ds{\prod}_{j}(\lambda-\gamma_j^{})}\cdot
\frac12\,\frac{d\gk}{\varrho_k^{}}\, \prod\limits_{j\ne
k}(\gk-\gamma_j^{})&&\quad\dashrightarrow\quad
\frac{-\mu}{z-\lambda}\,\frac{dz}{w}\,,
\end{alignat*}
wherein $\varrho_k^2=(\gk-E_1)\cdots(\gk-E_{2g+1})$.  Abelian integrals
of 3rd kind, as appeared in \eqref{final}, result from differential
equations \eqref{dubrovin}. Furthermore, substitution of expression
\eqref{final} into equation \eqref{1} leads, to get an identity, to
formula \eqref{trace}; this can serve as yet another way of derivation
of the trace formula. Symmetrizing  right hand sides of equations
\eqref{dubrovin}, we rewrite them down in a form that admits an
application of the indefinite integration operations, that is
\eqref{jac}. This set determines $\gk$ as functions of $x$.
\end{proof}

\begin{corollary}\label{C2}
In $u(x)$-representation the extension $\Nu\langle \Psi^{\pm}\rangle$
requires the one quadrature \eqref{pm}. If formula \eqref{final} is
used then extensions $\Nu\langle\Psi^{\pm}\rangle$ consist in an
adjunction of a symmetric sum of the logarithmic Abelian integrals.
\end{corollary}

\begin{remark}[{\sc Definition}]\label{R4}
Special attention must be given to the fact that integrability of the
`linear $\Psi$' is also algebraic (hyperelliptic) as is nonlinear
integrability of $\Nu$. More precisely, in what follows \textit{the
term `algebraic' will mean that ultimate answers contain finitely many
indefinite integrals of algebraic functions and inversions of the
formers}.
\end{remark}

As for transition from primary $\lambda$-dependence to the $x$-one and
vice versa, the duality between spectral and quadrature representations
is nontrivial; it is characterized by the following statement.

\begin{theorem}\label{T4}
Quadrature and spectral approaches are equivalent in the sense that the
explicit transition between formula \eqref{final} and its spectral
counterpart $($see below\/$)$ is realized through the Weierstrass
theorem on permutation of limits and parameters in a normalized Abelian
integral of  third kind.
\end{theorem}
\begin{proof}
Sub-exponential expression in \eqref{final} is a sum of Abelian
integrals, each with logarithmic singularities on $\R$ at two points:
$(z,w)=(\lambda,+\mu)$ and branch place $(z,w)=(\infty,\infty)$.
According to a Weierstrass theorem, we may exchange singularities of an
elementary logarithmic integral with its limits \cite{baker,clebsch,4}.
In our case, these are $(z,w)=(\gamma(x),\varrho(x))$ and a lower limit
$(z,w)=(\alpha,+\beta)$, where
$\beta^2=(\alpha-E_1)\cdots(\alpha-E_{2g+1})$. More precisely,
switching the places
$$
\left\{\begin{matrix}\big(\gamma(x),\varrho(x)\big)\\(\alpha,\beta)
\end{matrix}\right\}
\quad\rightleftarrows\quad \left\{
\begin{matrix}(\lambda,\mu)\\(\infty,\infty) \end{matrix}\right\},
$$
we obtain that the difference
$$
\int\limits_\alpha^{\;\gamma(x)}\!\!\!\!\!
\frac{w+\mu}{z-\lambda}\,\frac{d
z}{w}- \int_{\infty}\limits^{\;\lambda}\!
\left\{\frac{w+\varrho(x)}{z-\gamma(x)}-
\frac{w+\beta}{z-\alpha}\right\}
\frac{dz}{w}=\cdots
$$
is to be everywhere finite quantity, that is  certain holomorphic
integral:
$$
\cdots=\int\limits^{\;\gamma(x)}\!\!\!\!\!\!\!
\mbig[3]\{A_1(\lambda)+A_2(\lambda)\,z+\cdots+
A_g(\lambda)\,z^{g-1}\mbig[3]\}\,\frac{dz}{w}\,.
$$
Sum of $g$ such quantities must be a holomorphic integral depending
symmetrically on $\gamma$'s. Expression \eqref{final} is thus converted
to its dual object
\begin{equation}\label{dual}
\Psi^{\pm}\simeq\exp \frac{1}{2}
\mbig[10]\{\!\!
\int\limits^{\;\lambda}
\!\frac{w\pm\varrho_1^{}\!(x)}
{z-\gamma_{1}^{}\!(x)}\,
\frac{dz}{w}
\,\,\,\,+ \cdots + \int\limits^{\;\lambda}
\!\frac{w\pm\varrho_g^{}(x)}{z-\gamma_g(x)}\,\frac{dz}{w}
+\mathrm{holomorhic}(\lambda,x)\!\!\mbig[10]\}\,.
\end{equation}
This is nothing else but the spectral formula by Its \& Matveev
\cite{21}   deserving to be mentioned more often. We
reproduce\footnote{We have not found mention of this important result
in the literature. See also formula (3) in Akhiezer's work \cite{15}.}
their result as it has been written in \cite[p.~351]{21}:
\begin{equation}\label{a1}
\omega(\lambda)=\int\limits_{\beta_n}^{\;\lambda}\!\!
\frac{M(\lambda)}{2\,\sqrt{p(\lambda)}}\,d\lambda,\qquad
\int\limits_{\beta_{j-1}}^{\,\alpha_j}\!\!\!\!\!d\omega(\lambda)
=0,\qquad j=1,\ldots,n\,,
\end{equation}
where
$M(\lambda)=\lambda^n+a_1\lambda^{n-1}+\ldots+a_n$ and
\begin{equation}\label{a2}
\omega_k^{}(\lambda)=\int\limits_\infty^{\;\lambda}\!\!
\bigg(
\mfrac{\sqrt{P(\lambda)}+\sqrt{P(\lambda_k(x))}}{\lambda-\lambda_k(x)}
-\mfrac{\sqrt{P(\lambda)}+\sqrt{P(\lambda_k(0))}}{\lambda-\lambda_k(0)}
+M_k(\lambda)\bigg)\mfrac{d\lambda}{2\,\sqrt{\smash[b]{P(\lambda)}}}\,,
\end{equation}
\begin{equation}\label{a3}
\psi(x,\lambda)=\exp\!\bigg(\ri\,x\,\omega(\lambda)+
\sum\limits_{k=1}^n\omega_k(\lambda)\bigg).
\end{equation}
Meaning of all the quantities presented in \eqref{a1}--\eqref{a3} and
transition  \eqref{final} $\rightleftarrows$ \eqref{a3} are obvious
from the context. To put it differently, the permutation theorem of
Weierstrass, regarding inverse transition $\eqref{a3}\rightarrow
\eqref{final}$, is a way of doing normalization of periods of these
integrals so that all the  parameters in the spectral formulae
\eqref{a1}--\eqref{a3} can be `dumped' to a common multiplication
constant for the $\Psi$. By this means the \textit{non-indefinite}
integrals \eqref{dual} or \eqref{a2} with a parametrical dependence on
transcendental functions $\gk(x)$ turn into the  \textit{indefinite}
ones \eqref{final} of an algebraic function.
\end{proof}

Description of invariant property of equation \eqref{1} to be
integrable has been completed and we conclude the section with the
comments about fundamental difference between spectral and quadrature
modes of getting the formulae.

\subsection{On a Riemann surface}
At this point not only do analysis on Riemann surfaces does
not come into play, but also the surfaces themselves do not appear. One
has just a designation
$$
\mu\DEF\sqrt{(\lambda-E_1) \cdots (\lambda-E_{2g+1})}
$$
and the theory consists of elementary substitutions (see footnote on
p.~\pageref{kov}). On the other hand, verification of spectral formulae
\eqref{a1}--\eqref{a3} is a highly nontrivial task since they contain a
complete set of transcendental objects of Riemann's theory of Abelian
integrals and differentiation of an integrand containing $\gamma$'s. In
this respect, not using the permutation theorem, the `spectral
integral' in \eqref{dual}, that is
$$
\int\limits^{\;\lambda}_\infty
\!\frac{w\pm\varrho(x)}{z-\gamma(x)}\,\frac{dz}{w}\,,
$$
would be very akin to an integral representation of any special
function, say, complete elliptic Legendre's integral
$$
K(x)=\int\limits_0^{\;1}\!\!\!\!
\frac{dz}{\sqrt{(1-z^2)(1-x^2z^2)\,}}\,.
$$
The latter is not expressible by means of  any finite Liouvillian
extension over $\mathbb{C}(x)$ since $K(x^2)$ satisfies a 2nd order
irreducible ${}_2F_1\mbig[3](\frac12,\frac12;1\big|x^2
\mbig[3])$-hypergeometric equation \cite{ell}.

\begin{remark}[history]\label{R5}
Both the integrabilities are due to Liouville but chronologically,
`linear integrability' (1830--40s) was preceded by more famous
nonlinear integrability of Hamiltonian systems (1840--50s). Despite the
numerous modern literature, the fact that these two kinds of
integrability are non-casually related to the one name Liouville was
first observed by Morales-Ruiz \cite[pp.~51--52]{ispanec}.
\end{remark}

\section{The $\Theta$-functons\label{5}}
By virtue of the fact that extension $\Nu\mbig[0]\langle\Psi^{\pm}
\mbig[0]\rangle$ is transcendental, analytic representations for the
previous formulae require introducing new functions. These are the
$\Theta$-series \eqref{series} involved to the theory by Matveev and
Its in their famous work \cite{21}.

\begin{theorem}\label{T5}
The $\Theta$-function representations \eqref{theta},  \eqref{its} are
deducible from quadrature \eqref{final}. The expressions \eqref{theta}
and \eqref{final} are proportional to each other.
\end{theorem}

Proof of this theorem and consecutive derivation of formula
\eqref{theta} from \eqref{final} have been detailed in
\cite[\S\,7]{br3}. An important point here is the deductive appearance
of all the aggregates of  formula \eqref{theta} when it is viewed as an
axiomatic one: the integral $\II(\p)$, its periods $\boldsymbol{U}$,
and the Abel map $\boldsymbol{\mathfrak{U}}(\p)$.

From this theorem, we may draw the conclusion that when recognizing the
integrability of linear equations  the $\Theta$-functions themselves
are not necessary. They realize a step to be considered as the next one
after emergence of an integral symbol~$\int$ in \eqref{pm} and
\eqref{jac}.

\subsection{Some differential properties of theta-functions}
Theorem~\ref{T5} suggests a search for representation of the field
$\Nu$ by means of $\Theta$-functions. To do this require some
differential properties of $\Theta$-functions and we show further that
they are available. All the FG-theory tells us that Abelian and
BA-functions satisfy certain differential identities. By these
identities are meant the fact that Abelian functions, as theta-function
ratios of linear sections of $g$-dimensional jacobians of curves, have
a lot of differential relations between themselves and many of such
relations have forms of known integrable PDEs \cite{BEL, GH}. Adding
here exponential functions of the BA-type, we involve into analysis
$(\bL, \,\bA)$-pairs for these PDEs. Moreover, let FG-potential be
expressible through the $\theta$-functions of Jacobi. Then not only do
Abelian and BA-functions satisfy certain differential identities but
$\theta$-functions themselves also satisfy some ODEs.

Denote  by $\theta\AB{\varepsilon}{\delta}$  the standard
$\theta$-series of Jacobi with characteristics $(\varepsilon,\delta)$
\cite{ell}:
\begin{equation}\label{jacobi}
\theta\AB{\varepsilon}{\delta}(x|\tau)\DEF
\sideset{}{_k}\sum\limits_{-\infty}^{\infty}\!\!
\re_{\mathstrut}^{\pi \ri
\left(\!k+\frac\varepsilon2\!\right)^{\!2}\tau+ 2\pi
\ri\left(\!k+\frac\varepsilon2\!\right)
\left(\!z+\frac\delta2\!\right)}\,,
\end{equation}
\ie, $\theta\AB{1}{1}=-\theta_1$, $\theta\AB{1}{0}=\theta_2$,
$\theta\AB{0}{0}=\theta_3$, $\theta\AB{0}{1}=\theta_4$. Let
$\vartheta\DEF\theta(0|\tau)$ be corresponding $\vartheta$-constants
and $\Dtheta(x|\tau)$ stands for $x$-derivative  of the series
$-\theta\AB{1}{1}(x|\tau)$. Period of the meromorphic elliptic integral
is denoted by $\eta=\zeta(1|1,\tau)$.

\begin{theorem}\label{T6}
Jacobian functions $\theta\AB{\varepsilon}{\delta}$, $\Dtheta$ with
arbitrary integral characteristics are differentially closed over the
field of the $(\vartheta^2,\eta)$-constants and satisfy the autonomous
ODEs
\begin{equation}\label{new}
\left\{
\begin{aligned}
\frac{\partial\theta\AB{\varepsilon}{\delta}}{\partial x}&=
\frac{\Dtheta}{\theta\AB{1}{1}} \,
\theta\AB{\varepsilon}{\delta}-
(-1)^{\vcenter{\hbox{\Tiny$\Big[\!$}}\frac{\delta}{2}
\vcenter{\hbox{\Tiny$\!\Big]$}}
\varepsilon}_{\mathstrut}\,
\pi\,\vartheta\AB{\varepsilon}{\delta}^2\!\cdot\!
\frac{\theta\AB{\varepsilon\smin1}{0}\, \theta\AB{0}{\delta\smin1}}
{\theta\AB{1}{1}}\\[0.2em]
\frac{\partial\Dtheta}{\partial x}&=
\frac{\Dtheta^2}{\theta\AB{1}{1}}-\pi^2
\vartheta\AB{0}{0}^2\,\vartheta\AB{0}{1}^2\!\cdot\!
\frac{\theta\AB{1}{0}^{2\ds\mathstrut}}{\theta\AB{1}{1}}-
4\,\mbig[7]\{\eta+\frac{\pi^2}{12}\big(\vartheta\AB{0}{0}^4+
\vartheta\AB{0}{1}^4\big)\mbig[7]\} \!\cdot\!
\theta\AB{1}{1}
\end{aligned}\right.,
\end{equation}
where $\big[\frac{\delta}{2}\big]$ signifies an integer part of the
number $\delta/2$.
\end{theorem}

These formulae are  consequences of more general differential
properties of Jacobi's functions briefly tabulated in \cite{brP6}. One
can see that the similar properties are inherent characteristics of the
general $\Theta$-functions if the $g$-dimensional jacobian  is
isomorphic to a product of elliptic curves. Many examples of such
reductions can be found in \cite{38}.

\begin{example}
Define a $\Theta$-function with characteristics
$\hbox{\larger$\AB{\boldsymbol{\alpha}}{\boldsymbol{\beta}}$}$ as
follows:
\begin{equation}\label{thetaAB}
\Theta\hbox{\larger$\AB{\boldsymbol{\alpha}}{\boldsymbol{\beta}}$}
(\boldsymbol{z}|\Pi)\DEF
\,\ri^{\langle\boldsymbol{\alpha},\,\boldsymbol{\beta}\rangle}\,
\Theta\!\!\left(\boldsymbol{z}+ \mfrac12
\Pi\boldsymbol{\alpha}+\mfrac12\boldsymbol{\beta}\mbig|\Pi\right)\cdot
\re^{\pi\ri\left\langle\boldsymbol{\alpha},\,
\boldsymbol{z}+\frac14\Pi\boldsymbol{\alpha} \right\rangle}\,.
\end{equation}
Then in the case $g=2$ we have the following identity.
\end{example}
\begin{proposition}\label{P1}
The reduction formula for genus $g=2$ under $\Pi_{12}=\frac12$$:$
\begin{equation}\label{red}
\begin{aligned}
&\Theta\hbox{\larger$\AB{\alpha\,\varepsilon}{\beta\,\delta}$}
\mbig[7](\begin{smallmatrix}\frac12 z-\frac18\alpha\tau\\
{}^{\mathstrut} w-\frac14\alpha\end{smallmatrix}\Big| \ds
\begin{smallmatrix}
\frac14\tau&\frac12_{}\\\frac12\phantom{\tau}&
\varkappa\end{smallmatrix}\mbig[7])={}\\[0.5em]
&\qquad{}=
\re^{\frac{\pi}{2}\ri\alpha(z+\beta+\frac14\alpha\tau)}_{\mathstrut}
\!\cdot\!
\mbig\{\theta\AB{0}{\varepsilon} (z|\tau)\,
\theta\AB{\varepsilon}{\delta} (w|\varkappa)+
\ri^{2\beta+\varepsilon}_{\mathstrut}\!\cdot\!
\theta\AB{1}{\varepsilon}
(z|\tau)\, \theta\AB{\varepsilon}{\delta\smin1} (w|\varkappa)
\mbig\}.
\end{aligned}
\end{equation}
Differential properties of these $\Theta,\,\Theta'$-functions follow
completely from Theorem~$\ref{T6}$.
\end{proposition}

This example is  a rather illustrative one because  curves have often
symmetries and if a genus-2 curve has an involutory symmetry differed
from the hyperelliptic one $(\lambda,\mu)\mapsto(\lambda,-\mu)$ then
one can show that its $\Pi$-matrix is reducible to form \eqref{red}.
Formula \eqref{red} is perhaps a simplest case of reduction of the
theta-functions to the two elliptic tori $\tau$ and $\varkappa$. More
complex equivalent of \eqref{red} is presented in \cite{38}.

\begin{corollary}\label{C3}
Let jacobian of the curve \eqref{hyper} be splittable  into a product
of the elliptic curves and the collective symbols $\theta$,
$\vartheta$, $\eta$ stand for arising Jacobi's theta-functions and
their constants. Then the field
$\mathbb{C}_\partial(\theta;\vartheta^2,\eta,\ldots)$ is a differential
extension of $\Nu$ $($dots indicate other constants of the field\/$)$.
\end{corollary}

\begin{example}
Non-elliptic 2-gap potential  \eqref{its}  with
$\Theta=\Theta\AB{00}{00}$ for the reduction case \eqref{red}:
\begin{equation}\label{2gap}
u=-2\ln_{\mathit{xx}}\!\!\mbig\{\theta_4(Ux+A|\tau)\,
\theta_2(Vx+B|\varkappa)-\ri\,
\theta_1(Ux+A|\tau)\,\theta_1(Vx+B|\varkappa)\mbig\},
\end{equation}
where  $\{\tau,\varkappa,A,B,V\}$ are arbitrary.  Since the
differential $\theta$-calculus is completely at hand, we get a
particular but nontrivial example of solution to Dubrovin's
effectivization formulae for genus $g=2$. Recall that the problem
consists \cite{5} in determination of the `winding' vector
$\boldsymbol{U}$ and is described by a system of equations containing
the undetermined fourth derivatives of the $\Theta$-function \cite{5}.
In the example under consideration the ultimate answer turns out to be
quite finite but somewhat large to display here. Equation for the one
sought-for quantity $U$ is an algebraic equation of degree 9 (exercise:
derive it).
\end{example}

\subsection{Linearly exponential divisor and
$\Theta$-representation of $\Nu$} Let us use notation of  formula
\eqref{theta} and plug into \eqref{its} an inessential exponential
multiplier:
\begin{equation}\label{k}
u(x)=-2\,\frac{d^2}{dx^2}\ln\Theta(x\,\boldsymbol{U}+\boldsymbol{D})\,
\re^{h\,x}+\mathrm{const}\,.
\end{equation}
Introduction of this term is motivated by the fact that solutions for
the $\Psi$-function \eqref{theta} are expressed not only through the
$\Theta$-functions but involve an exponential factor. We shall call the
quantity $\Theta(x\,\boldsymbol{U}+\boldsymbol{D})\,\re^{hx}$, with
$h$, $\boldsymbol{U}$, and $\boldsymbol{D}$ being constants with
respect to $\partial$, the \textit{linearly exponential divisor} or
$\boldsymbol{\mathfrak{L}}_x$-divisor. Let us form a field
$\mathbb{C}(\boldsymbol{\mathfrak{L}}_x)$. The following proposition
gives a weaker property than Corollary~\ref{C3}, but it is valid for
arbitrary genera.

\begin{proposition}\label{P2}
The field $\mathbb{C}\big(\boldsymbol{\mathfrak{L}}_x\big)$ is
$\partial$-differential and finitely generated.
\end{proposition}
\begin{proof}
Expression \eqref{k} satisfies a Novikov equation which has finite
order $2g+1$. Hence the derivatives
$\frac{d^n}{dx^n}\Theta(x\,\boldsymbol{U}+\boldsymbol{D})\,\re^{hx}$ of
order $n\geqslant (2g+1)+2$ are expressed rationally through
$\boldsymbol{\mathfrak{L}}_x$ and its lower derivatives.
\end{proof}

The field $\Theta_\partial=
\mathbb{C}_\partial\big(\Theta(x\,\boldsymbol{U}+\boldsymbol{D})\big)$
may be considered as a field generated by one linear divisor ($h=0$).
It is obviously that $\Nu\subset\Theta_\partial$. Of course,
$\Theta_\partial$ contains now not only Abelian functions but this
extension is well defined since  $\Theta$-series and
$\boldsymbol{\mathfrak{b}}$-periods  $\boldsymbol{U}$ are computed once
$u(x)$ has been given. For this reason, we can redefine equation
\eqref{1} as one given over $\Theta_\partial$ and thereby we let
\begin{equation}\label{NT}
\NTheta\DEF\mathbb{C}_\partial
\big(\Theta(x\,\boldsymbol{U}+\boldsymbol{D})\big)\,.
\end{equation}
Although $\Nu\subsetneqq\NTheta$, the field $\NTheta$ is said to be a
$\Theta$-\textit{representation} of Novikov's fields. The constant
$\lambda$, for the moment, may be disregarded.

It should be emphasized that in the FG-integration there appears not
merely an abstract multi-dimensional $\Theta$-series but its
specification $\boldsymbol{\mathfrak{L}}_x$; the 1-dimensional linear
section of jacobians. Moreover, the most general $\Theta$-function is
an object defined up to an exponential multiplier (see \eqref{thetaAB})
so that we may think of it, and therefore of the \textit{divisor
$\boldsymbol{\mathfrak{L}}_x$, as a continual generalization of
$\Theta$-function with discrete characteristics}. Thus, Liouvillian
solutions $\Nu\langle\Psi^{\pm}\rangle$ are expressed through the
$\Theta$ with a \textit{linear} dependence of its arguments upon $x$
and the `linearly exponential' multiplier\footnote{This `linear
exponent' is a result of the contemporary theory; Baker \cite{32} did
not specify an exponential $\Theta$-structure of solutions but just
cited to pp.~275, 289 of his \cite{baker}.} $\re^{\II(\p)\,x}$ first
explicitly appeared in Akhiezer's work \cite{15} and subsequently was
axiomatized by Krichever \cite{24}.

\begin{remark}\label{R6}
In what follows  we shall exhibit that the continually parametric
object $\boldsymbol{\mathfrak{L}}_x$ may be introduced in its own
right, in a particular case,  through the \textit{quadrature
integrable} ODEs. Interestingly, the different kind sections of
theta-function arguments can lead to other important equations.
One remarkable property of such a kind appears even  in the $g=1$ case.
Let us consider the function $\theta(x|\tau)$ as a function on a
simplest (\ie, straight line) section of the 2-dimensional variety
\{jacobian $\otimes$ moduli space\}. Without loss of generality we may
impart to this function the form $\theta_1(A\tau+B|\tau)$. Then this
object generates the general Hitchin class of solutions to the sixth
Painlev\'e equation in a form exactly as does the finite-gap formula
\eqref{its}, that is logarithmic derivative of a ratio of entire
functions \cite{brP6}.
\end{remark}

\section{Integration as a linearly exponential
$\Theta$-extension\label{6}}

In this section we give a formulation of the $\Theta$-function scheme
as integration \`a la Picard--Vessiot. Let us consider the above
Picard--Vessiot extension \cite{singerput} $\Nu\subset
\Nu\langle\Psi^{\pm}\rangle$ in the representation $u(x)$. This
transcendence is Liouvillian and dependence of the $\Psi$ on parameter
$\lambda$ is also transcendental contrary to the `rationality' of the
field $\mathbb{C}(\lambda,\mu)$. Meanwhile, based on Theorems~\ref{T1}
and \ref{T5}, we see that the field has the following structure:
$$
\Nu\langle\Psi^{\pm}\rangle=
\mathbb{C}_\partial\!\bigg(
\Mfrac{\Theta\big(x\,\boldsymbol{U}+\boldsymbol{D}+
\boldsymbol{\mathfrak{U}}(\p)\big)}{
\Theta(x\,\boldsymbol{U}+\boldsymbol{D})}\,
\re^{\II(\p)\,x}\bigg).
$$
It immediately follows that if we pass to the $\Theta$-representation
\eqref{NT} then integration procedure can be reduced to one operation.
Namely, a field, over which an equation has been defined, is
supplemented with an element of the same form as one generating the
field itself:
\begin{equation}\label{emb}
\Nu\subset\NTheta \subset \mathbb{C}_\partial
\mbig[5](
\Theta(x\,\boldsymbol{U}+\boldsymbol{D})\, \re^{h\,x\ds\mathstrut},\,
\Theta(x\,\boldsymbol{U}+\boldsymbol{D}+
\boldsymbol{\mathfrak{U}}(\p))\,
\re^{\II(\p)\,x}\mbig[5])\,.
\end{equation}

It is significant in this viewpoint that  problem of the `linear
integration' drops out along with  the problem of building the base
field $\Nu$ being treated as a problem of the nonlinear integration.
In the $\Theta$-representation the potential is determined only by
means of operations in the field $\NTheta$; formula \eqref{k}.

\begin{theorem}\label{T7}
All the embeddings
$$
\Nu\subset\NTheta \subset\Nu \langle\Psi^{\pm}\rangle\subset
\mathbb{C}_\partial\mbig[5](
\Theta(x\,\boldsymbol{U}+\boldsymbol{D})\,,\,
\Theta\big(x\,\boldsymbol{U}+\boldsymbol{D}+
\boldsymbol{\mathfrak{U}}(\p)\big)\,
\re^{\II(\p)\,x}\mbig[5])
$$
are the Liouvillian extensions.
\end{theorem}
\begin{proof}
The fact that  embedding $\Nu\subset\NTheta$ is Liouvillian follows
directly form formula \eqref{its}. The statement concerning the second
embedding  is a consequence of \eqref{pm} because
$R(\dpotu;\lambda)\in\NTheta$. Rewriting formula \eqref{theta} in the
form
$$
\Theta\big(x\,\boldsymbol{U}+\boldsymbol{D}+
\boldsymbol{\mathfrak{U}}(\p)\big)\,
\re^{\II(\p)\,x}=\Psi^+\cdot
\Theta(x\,\boldsymbol{U}+\boldsymbol{D})\,,
$$
we deduce that property for the last embedding to be Liouvillian
results from the proportionality of \eqref{theta} and expression
\eqref{pm} for $\Psi^+$ (Theorems~\ref{T3} and \ref{T5}).
\end{proof}

\begin{theorem}\label{T8}
For generic $\lambda$  integration of equation \eqref{1} in the
$\Theta$-rep\-re\-sen\-ta\-tion is equivalent to a multiplication of an
element $\Xi(x)$ generating the field $\NTheta$ by an adjoined linearly
exponential divisor:
\end{theorem}
\begin{proof}
Consider $\Xi(x)=\Theta(x\,\boldsymbol{U}+\boldsymbol{D})^{\smin1}$. It
is clear that $\mathbb{C}_\partial(\Xi)=\NTheta$. Then
\begin{equation}\label{lin} \Psi^{\pm}\big(x;\lambda(\p)\big)=
C_{\pm}\cdot\Xi(x)\cdot \Theta\big(x\,\boldsymbol{U}+\boldsymbol{D} \pm
\boldsymbol{\mathfrak{U}}(\p)\big) \, \re^{\pm\II(\p)\,x}
\end{equation}
because all the holomorphic/meromorphic integrals on hyperelliptic
curves change sign under  permutation of sheets $\mu\mapsto -\mu$; we
may write $\pm \boldsymbol{\mathfrak{U}}(\p)$, $\pm\II(\p)$ in
\eqref{lin}.
\end{proof}

This theorem has an important treatment:
\begin{itemize}
\item \textit{When passing to the $\Theta$-representation the equation
\eqref{1} is integrated as if it were an equation with
\textbf{constant} coefficients. Integration procedure is thus
trivialized under a proper choice of `domain of rationality'. The
inverse transform method for the  soliton class is a particular
case of this general construction}.
\end{itemize}
Indeed, the simplest FG-case corresponds to the 0-gap one with
$\NTheta=\mathbb{C}(\lambda)$ and we need only one linear exponent; the
construction \eqref{lin} acquires the form
\begin{equation}\label{0}
\Psi^{\pm}(x;\lambda)= C_{\pm}\cdot \Xi(x)\cdot \re^{\pm
a(\lambda)x}\,, \qquad \Xi(x)=\re^{0\cdot
x}_{}\in\NTheta\,,
\end{equation}
wherein  $\Xi(x)$ has  the `same form' as the adjoint element
$\re^{a(\lambda)x}_{}$. Adjoining all the exponents associated with an
$N$-soliton solution (and their varieties like positons \eqref{positon}
or rational solitons), we obtain the general 0-gap case.

\begin{remark}\label{R7}
The structure of solution \eqref{lin} in form of simple multiplication
of elements generating $\NTheta$ and its extension
$\Nu\langle\Psi^{\pm} \rangle$ is not quite typical for equations
integrable by attaching the linear exponents \cite{kov} or, especially,
for equations with a solvable Galois group \cite{kolchin}. This
property owes its origin to the availability of $\lambda$ in equation
(Theorem~\ref{T1}).
\end{remark}

Transition between  $u$- and $\Theta$-representations is transcendental
with respect to  $\lambda$-dependence and other constants of the field.
These constants are the $\boldsymbol{\Pi}$-matrices of curves,
$\boldsymbol{U}$-periods, and vector $\boldsymbol{D}$. We may therefore
trivialize the scheme above if we proceed  further and redefine
equation \eqref{1} over \eqref{NT} as one given over the
$\lambda$-pencil (field) of the
$\boldsymbol{\mathfrak{L}}_x(\p)$-divisors:
\begin{equation*}\label{LP}
\boldsymbol{\mathfrak{L}}_x(\p)\DEF
\Theta\big(x\,\boldsymbol{U}+\boldsymbol{D}+
\boldsymbol{\mathfrak{U}}(\p)\big) \,
\re^{\II(\p)\,x}\,.
\end{equation*}
Although  field
$\mathbb{C}_\partial\big(\boldsymbol{\mathfrak{L}}_x(\p)\big)$ is
generated by the infinite number of elements, equation \eqref{1} itself
constitutes an infinite $\lambda$-pencil of equations. Strictly
speaking, both spectral and quadrature approaches require to look upon
Eq.~\eqref{1} as being a \textit{differentially-algebraic} one:
differential in $x$ and algebraic (polynomial) in $\lambda$.
Furthermore, by virtue of Proposition~\ref{P2}, the arbitrary
$\boldsymbol{\mathfrak{L}}_x(\p)$-divisor generates some solution of
Novikov's equation. It may be fixed by choice of one element of the
pencil $\boldsymbol{\mathfrak{L}}_x(\p)$:
\begin{equation}\label{Lo}
\NTheta\to\mathbb{C}_\partial\mbig[5](
\Theta\big(x\,\boldsymbol{U}+\boldsymbol{D}_{\sss0}+
\boldsymbol{\mathfrak{U}}(\p_{\sss0})\big) \,
\re^{\II(\p_{\sss0})\,x}\mbig[5])\,.
\end{equation}
Having extended the field \eqref{Lo}, that is $\mathbb{C}_\partial\big(
\boldsymbol{\mathfrak{L}}_x(\p_{\sss0})\big)$, to the field
$\mathbb{C}_\partial\big(\boldsymbol{\mathfrak{L}}_x(\p)\big)$, its
Galois group becomes trivial since the integral of equation \eqref{1}
is given now in form of a ratio of two field elements.

\section{Integrability and differential closedness\label{7}}

\subsection{Differential closure in terms of $\theta$-functions}
Attaching the divisor $\boldsymbol{\mathfrak{L}}_x(\p)$ as
transcendental element with a parameter $\p$ tells us that \textit{it
should be introduced to the theory as the base function}, along with
the available $\Theta$'s without parameter $\p$. Owing to
Theorem~\ref{T8} this  would arrive us at a closed differential
apparatus (differential closedness) accompanying  spectral equation and
potential. Presently, the general $\Theta$-formula realization of this
viewpoint is an open problem, which is why we illustrate it by cases
when $g$-dimensional $\Theta$-function reduces to a combination of
Jacobian ones.

At first glance, from  Theorem~\ref{T6}, it would seem that supplement
of the basis \eqref{new} with $\boldsymbol{\mathfrak{L}}_x$-divisor of
the type $\theta_1(x-\mathfrak{u})\,\re^{hx}$ requires also adjunction
of all the functions $\Dtheta,\,\theta_{2,3,4}(x-\mathfrak{u})$. That
no such complication takes place will be apparent from the following
statement.

\begin{theorem}\label{T9}
For  Weierstrassian curve \eqref{torus} with modulus
$\tau=\omp\!\!/\omega$ one defines an elliptic
$\boldsymbol{\mathfrak{L}}_x(\p)$-divisor $\Lambda$ by the formula
$$
\Lambda(x;\mathfrak{u}|\tau)\DEF\theta_1(x-\mathfrak{u}|\tau)\,
\exp\!\!\mbig[7](\Mfrac{\Dtheta(\mathfrak{u}|\tau)}
{\theta_1(\mathfrak{u}|\tau)}\,
x+h\,x\mbig[7]),\qquad
\mathfrak{u}\notin \mathbb{Z}\,\tau+\mathbb{Z}\,,
$$
where $h$ is an extra parameter. Then the six functions $\Lambda$,
$\Dtheta$, and $\theta_k$ $(k=1,2,3,4)$ satisfy the closed autonomous
system of ODEs over the field $\mathbb{C}\big(\eta,\vartheta^2,
\theta(\mathfrak{u})\big)$$:$
\begin{align}
& \label{!}
\left\{
\begin{aligned}
\frac{\partial\theta_k}{\partial x} &=
\frac{\Dtheta}{\theta_1} \,\theta_k-
\pi\,\vartheta_k^2\!\cdot\!
\frac{\theta_n\,\theta_m}{\theta_1}\,,\qquad
n=\frac{8\,k-28}{3\,k-10}\,,\quad
m=\frac{10\,k-28}{3\,k-8}\\[0.3em]
\frac{\partial\Dtheta}{\partial x}&=
\frac{\Dtheta{}^2}{\theta_1}-\pi^2\,\vartheta_3^2\,\vartheta_4^2
\!\cdot\! \frac{\theta_2^{2}}{\theta_1}-
4\,\mbig[7]\{\eta+\frac{\pi^2}{12}\big(\vartheta_3^4+\vartheta_4^4\big)
\mbig[7]\} \!\cdot\!\theta_1
\\[0.3em]
\frac{1}{\Lambda}\frac{\partial\Lambda}{\partial x}&=
\frac{\Dtheta}{\theta_1}+
\frac{\pi\,\vartheta_2^2}{\theta_1(\mathfrak{u})}\!\cdot\!
\frac{\theta_1^3(\mathfrak{u})\!\cdot\!\theta_2\,\theta_3\,\theta_4+
\theta_2(\mathfrak{u})\,\theta_3(\mathfrak{u})\,\theta_4(\mathfrak{u})
\!\cdot\!\theta_1^3}
{\theta_1\!\cdot\!\big(\theta_2^2(\mathfrak{u})\!\cdot\!\theta_1^2-
\theta_1^2(\mathfrak{u})\!\cdot\!\theta_2^2\big)}+h
\end{aligned}
\right.\,.
\end{align}
\end{theorem}

Motivation for the theorem lies in the fact that Abelian integrals, on
the one hand, are representable in terms of theta-functions and, on the
other hand,  are differentially closed: the base integrals of 1st, 2nd,
and 3rd kind form a differential basis. Indeed, derivatives of
integrals are functions and any function is expressed through the two
base ones $(\wp,\,\wpp)$ (generators of an elliptic functions field)
which are in turn integrals of exact meromorphic differentials of 2nd
kind. However theta-function representation for a 3rd kind integral is
still lacking in system \eqref{new}.

\begin{proof}
Meromorphic functions are at hand since they are formed by
$\theta$-quo\-ti\-ents.  Derivatives of meromorphic functions are again
meromorphic ones but differentiation of $\theta_k(x)$ generates
$\Dtheta(x)$. The quotient $\Dtheta\,/\theta_1$ is proportional to the
Weierstrass $\zeta$-function which in turn represents a meromorphic
(\ie, 2nd kind) elliptic integral \cite{ell}; it is alone as genus
$g=1$. Without loss of generality we may set that the missing integral
of 3rd kind has two logarithmic singularities. The place of one of them
may be fixed at $x=0$ and the second place can be taken as a parameter.
Call it $\mathfrak{u}$. Residues of a corresponding differential are
opposite in sign and they can be moved to a common multiplication
constant. Such an integral $\III$ is unique (even for arbitrary $g$) up
to a holomorphic one(s) since all the other integrals are expressed
through $\III=\III(x;\mathfrak{u})$. Its derivatives are again
meromorphic functions and the process closes. It will suffice to add an
exponent of $\III$ and  integral $\III$ itself is given by the known
formula
\begin{equation}\label{III}
\III(x;\mathfrak{u})\DEF
\frac12\int\limits^{\;z}\!\frac{w+w_{\sss0}}{z-z_{\sss0}}
\frac{dz}{w}=
\ln\!\frac{\sigma(2\,x-2\,\mathfrak{u})}{\sigma(2\,x)}\,
\re^{2\,\zeta(2\,\mathfrak{u})\,x}_{\mathstrut}\,,
\end{equation}
where $z=\wp(2\,x),\,w=\wpp(2\,x)$ and
$z_{\sss0}=\wp(2\,\mathfrak{u}),\, w_{\sss0}=\wpp(2\,\mathfrak{u})$.
Weierstrassian parameters $(\omega,\omp)$, presented in \eqref{III},
are replaced by the one quantity $\tau$ thanks to homogeneity relation
$\omega^2\wp(\omega\,x|\omega,\omp)=\wp(x|1,\tau)\FED\wp(x|\tau)$.
Holomorphic integral is absent in system \eqref{new} but is present in
a basis of Abelian integrals. Missing element $\Lambda_{\sss0}$ can be
formally adjoined by setting $\Lambda_{\sss0}(x|\tau)=x$ and
supplementing system \eqref{!} with equation
$\frac{d\Lambda_{\sss0}}{dx}=1$. Adding to \eqref{III} the holomorphic
integral $h\,\Lambda_{\sss0}$ and converting the right hand side of
\eqref{III} to the $\theta$-functions, one obtains that adjunction of
$\exp(\III)$ is equivalent to adjunction of the object
$\Lambda(x;\mathfrak{u}|\tau)$. Differentiating \eqref{III} and
converting it to the $\theta$'s, one arrives, after some
simplification, at the last equality in Eqs.~\eqref{!}.

If $\mathfrak{u}\in \mathbb{Z}\,\tau+\mathbb{Z}$ then the integral
$\III$ turns into a logarithm of meromorphic function: $\ln(\wp(x)-e)$.
There is nothing to adjoin.
\end{proof}

\begin{corollary}\label{C4}
Let potential be a finite-gap one and $g$-dimensional jacobian  split
into a product of the elliptic curves. Then differentiation of
$g$-dimensional functions
$\Theta\big(x\,\boldsymbol{U}+\boldsymbol{D}\big)$ reduces to a set of
the `$1$-dimensional' equations \eqref{!}. The functions
$\theta_k,\,\Dtheta$ and $\Lambda$, taken possibly with different
moduli $\tau$, form the differentially closed basis over which every
Novikov's equation of order $\leqslant 2g+1$ and  problem \eqref{1} are
integrated.
\end{corollary}

In the framework of this corollary integrability of Novikov's equations
is a manifestation of differential closedness of the first two
equations in \eqref{!}. The third equation in \eqref{!} `integrates'
equations with a parameter and their consequences (see examples in
\cite{br3}). Here, `integrates' means that all these equations are
nothing more than combinations of system \eqref{!} and its derivatives.
Let us consider some examples.

\begin{example}
The two gap Lam\'e potential
$$
\Psi_{\mathit{\!xx}}=\big\{24\,\wp(2\,x|\tau)+\lambda\big\}\,\Psi\,.
$$
This is a classical example presented in many places \cite{38,GH}.
Corresponding solution expressed in terms of our objects reads as
follows:
\begin{equation}\label{6P}
\Psi^{\pm}=\frac{d}{dx}\frac{\Lambda(x;\pm\mathfrak{u}|\tau)}
{\theta_1(x|\tau)}\,,\quad h=\frac{\pm2\,\mu}
{3\,\lambda^2-12^2\,g_2^{}(\tau)}\,,\quad
\wp(2\,\mathfrak{u}|\tau)=\frac{\lambda^3+12^3\,g_3^{}(\tau)}
{36\,\lambda^2-12^3\,g_2^{}(\tau)}\,,
\end{equation}
$$
\mu^2=\bigl(\lambda^2-48\,g_2^{}(\tau)\big)
\big(\lambda^3-36\,g_2^{}(\tau)\,\lambda+ 432\,g_3^{}(\tau)\bigr)
\quad\big(\Leftrightarrow\eqref{hyper}\big)\,,
$$
where $g_2^{}(\tau)$, $g_3^{}(\tau)$ are the standard modular
$\tau$-representations for Weierstrassian parameters $a$, $b$ entering
into the curve $w^2=4\,z^3-a\,z-b$ \cite{ell}.  Novikov's equation
\eqref{N2} is satisfied under $(c_1^{},c_2^{})=(0,-672\,g_2^{}(\tau))$
and $R$-polynomial has the form
$$
R(\dpotu;\lambda)=\lambda^2-\frac12\,u\,\lambda+
\frac14\,u^2-36\,g_2^{}(\tau)\,,
$$
where $u=24\,\wp(2\,x|\tau)$.
\end{example}

\subsection{Non-elliptic example}
We consider here  a potential being no elliptic function but
expressible through the `elliptic' $\theta$-functions.

\begin{example}
The $\Psi$-function for non-elliptic potential \eqref{2gap}. It is a
nonlinear superposition of the one-gap $\Psi$-functions. If we denote
for brevity $z=Ux+A$ and $w=Vx+B$ we then derive that
\begin{equation}\label{psi2}
\Psi(x;\lambda)=
\frac{\Lambda\big(z{+}\frac12\,\tau;\mathfrak{U}_1\big|\tau\big)\,
\Lambda\big(w{+}\frac12;\mathfrak{U}_2\big|\varkappa\big)-
\Lambda(z;\mathfrak{U}_1|\tau)\, \Lambda(w;\mathfrak{U}_2|\varkappa)
}{ \theta_1\!\big(z{+}\frac12\,\tau\big|\tau\big)\,
\theta_1\!\big(w{+}\frac12\big|\varkappa\big)- \theta_1(z|\tau)\,
\theta_1(w|\varkappa)}\,\re^{\II(\lambda)x}_{\mathstrut}\,,
\end{equation}
where holomorphic integrals $\mathfrak{U}_1$, $\mathfrak{U}_2$ are the
certain linear combinations
$\mathfrak{U}_k=C_{\mathit{kj}}\mathfrak{u}_j$ of the elliptic
holomorphic ones $\mathfrak{u}_1$, $\mathfrak{u}_2$ because the curve
\eqref{hyper} corresponding to the $\Theta$-function \eqref{red} has
the form
\begin{equation}\label{P5}
\mu^2=\lambda(\lambda-1)(\lambda-\boldsymbol{a})
(\lambda-\boldsymbol{b}) (\lambda-\boldsymbol{ab})\FED P_5(\lambda)
\end{equation}
and is realized as covers of  two tori defined by moduli $\tau$ and
$\varkappa$:
\begin{equation}\label{cover}
\begin{aligned}
\wp(2\,\mathfrak{u}_1|\tau)+\vartheta_2^4(\tau)+
\vartheta_3^3(\tau)&=
3\,\frac{(1-\boldsymbol{a})(1-\boldsymbol{b})\lambda}
{(\lambda-\boldsymbol{a})(\lambda-\boldsymbol{b})}\,
\vartheta_3^4(\tau)\,,\\[0.3em]
\wp(2\,\mathfrak{u}_2|\varkappa)+
\vartheta_2^4(\varkappa)+\vartheta_3^3(\varkappa)&=
3\,\frac{(1-\boldsymbol{a})(1-\boldsymbol{b})
\lambda}{(\lambda-\boldsymbol{a})(\lambda-\boldsymbol{b})}\,
\vartheta_3^4(\varkappa)\,.
\end{aligned}
\end{equation}
These are the classical formulae by Jacobi \cite{38} presented in terms
of Weierstrass'  $\wp$ and the pair of branch points
$\{\boldsymbol{a},\boldsymbol{b}\}$ and moduli $\{\tau,\varkappa\}$ are
commonly written down for one another \cite{38}. All the information
concerning this curve can be found in \cite{38} and we omit details of
some calculations. We need to compute the integral $\II(\lambda)$.

Abelian integrals for the curve \eqref{P5} are expressed through
$\theta$, $\Dtheta$ and therefore derivation of the meromorphic
integral $\II(\lambda)$ is a routine calculation. An explanation is
that the reduction of holomorphic integrals to elliptic ones entails
the reduction of the meromorphic Abelian integrals to the meromorphic
elliptic ones. The latter  are expressed through Weierstrassian
$\zeta$-function, \ie, $\theta'\!/\theta$, and meromorphic elliptic
functions. First translate formulae for cover \eqref{cover} into the
language of $\theta$-functions:
\begin{equation}\label{ell}
\frac{\vartheta_2^2(\tau)}{\vartheta_3^2(\tau)}
\frac{\theta_4^2(\mathfrak{u}_1|\tau)}
{\theta_1^2(\mathfrak{u}_1|\tau)}=
\frac{(1-\boldsymbol{a})(1-\boldsymbol{b})\lambda}
{(\lambda-\boldsymbol{a})(\lambda-\boldsymbol{b})}\,,\qquad
\frac{\vartheta_2^2(\tau)}{\vartheta_3^2(\tau)}
\frac{\theta_4^2(\mathfrak{u}_1|\tau)}
{\theta_1^2(\mathfrak{u}_1|\tau)}=
\frac{\vartheta_2^2(\varkappa)}{\vartheta_3^2(\varkappa)}
\frac{\theta_4^2(\mathfrak{u}_2|\varkappa)}
{\theta_1^2(\mathfrak{u}_2|\varkappa)}\,.
\end{equation}
This is a complete set of equations determining the holomorphic
integrals $\mathfrak{u}_1$, $\mathfrak{u}_2$ as functions of $\lambda$.
The point $\lambda=\infty$ corresponds to the values
$\mathfrak{u}_1=\frac12\,\tau$ and $\mathfrak{u}_2=\frac12\,\varkappa$.
Drop out for the moment indication of modulus $\tau$ in the following
transformation of a meromorphic elliptic integral:
$$
\begin{aligned}
\int\!\!\!\!\frac{s\,ds}{\sqrt{4\,s^3-g_2^{}s-g_3^{}}}&= \left| s=
\frac{\pi^2}{12}\!\left(\vartheta_3^4+\vartheta_4^4-
3\,\vartheta_3^2\,\vartheta_4^2
\,\frac{\theta_3^2(\mathfrak{u}_1)} {\theta_4^2(\mathfrak{u}_1)}
\right)\right|\\[0.3em]
&=-\frac12\,\frac{\Dtheta\!
\big(\mathfrak{u}_1{-}\frac12\,\tau\big)}
{\theta_1\!\big(\mathfrak{u}_1{-}\frac12\,\tau\big)}
-2\,\eta\,\mathfrak{u}_1=\cdots
\end{aligned}
$$
On the other hand, the 1st formula in \eqref{ell} supplemented with use
of the standard quadratic $\vartheta,\,\theta$-identities brings this
integral into the following expression:
$$
\cdots=\mathrm{const}\int\!\!s\!\cdot\!
\frac{\lambda-\sqrt{\boldsymbol{a}\mathstrut}
\sqrt{\boldsymbol{b}}}{\sqrt{P_5(\lambda)}}\,d\lambda
\,\simeq\,\int\!\!\left(
\frac{(\lambda-\boldsymbol{a})(\lambda-\boldsymbol{b})}
{(1-\boldsymbol{a})(1-\boldsymbol{b})\lambda}+
\frac{\vartheta_2^4+\vartheta_3^4}{3\,\vartheta_3^4}\right)\!
\frac{\lambda-\sqrt{\boldsymbol{a}\mathstrut}
\sqrt{\boldsymbol{b}}}{\sqrt{P_5(\lambda)}}\, d\lambda\,,
$$
that is  meromorphic integral on the curve \eqref{P5} with a pole at
$\lambda=\infty$ and a surplus one at $\lambda=0$. Doing the same for
the second torus $(\mathfrak{u}_2)$ with modulus $\varkappa$, we obtain
one more meromorphic integral with the same infinities. Forming their
linear combination, we can construct the integral with a single
singularity at infinite point. After some algebraic simplifications the
sought-for result becomes:
\begin{equation}\label{mero}
\begin{aligned}
\II(\lambda)&=a\!\cdot\!
\frac{\Dtheta\!\big(\mathfrak{u}_1{-}\frac12\,\tau\big|\tau\big)}
{\theta_1\!\big(\mathfrak{u}_1{-}\frac12\,\tau\big|\tau\big)} +
b\!\cdot\!
\frac{\Dtheta\!\big(\mathfrak{u}_2{-}\frac12\,
\varkappa\big|\varkappa\big)}
{\theta_1\!\big(\mathfrak{u}_2{-}\frac12\,
\varkappa\big|\varkappa\big)}+
c\!\cdot\!\mathfrak{u}_1+d\!\cdot\!\mathfrak{u}_2\\[0.3em]
&= a\!\cdot\!
\frac{\Dtheta(\mathfrak{u}_1|\tau)}
{\theta_1(\mathfrak{u}_1|\tau)} +
b\!\cdot\! \frac{\Dtheta(\mathfrak{u}_2|\varkappa)}
{\theta_1(\mathfrak{u}_2|\varkappa)}+
\frac{(\lambda+p)\,\mu}{\lambda\,(\lambda-\boldsymbol{a})
(\lambda-\boldsymbol{b})}+ q\!\cdot\!\mathfrak{u}_1+
r\!\cdot\!\mathfrak{u}_2\,,
\end{aligned}
\end{equation}
where constants $\{a,b,c,d,p,q,r\}$ depend only  on parameters of the
potential \eqref{2gap}, \ie, on $\{\tau$, $\varkappa$, $A$, $B$, $V\}$,
and are independent of $\lambda$. Again, after careful co-ordination of
all the moduli and normalizing constants the direct check of \eqref{1},
\eqref{2gap}, \eqref{psi2}, and \eqref{mero} becomes a good exercise in
a differential $\theta$-calculus. We mention in passing that this
example cannot be elaborated in the framework of the standard elliptic
soliton theory \cite{38}.
\end{example}

\subsection{A new treatment of the spectral parameter}
In  proof of  Theorem~\ref{T9} parameter $\mathfrak{u}$ was the only
`external' parameter independent of `internal' parameters of the curve
(moduli). On the other hand, the only parameter being external to the
field $\Nu$ and equation \eqref{1} is $\lambda$. It plays an isolated
role. Apart from the fact that it is merely present in equation, it may
be treated as an object arising from the two independent `mechanisms':
1) adjunction of a transcendental element and, on the other hand, 2)
differential closedness of all the Abelian integrals. As evidenced by
the foregoing and formula \eqref{final}, these are the same things:

\begin{itemize}
\item \textit{The logarithmic singularity in a canonical integral
of third kind is  arbitrary and independent of moduli. The property
of the theory to be integrable is, by construction, independent of
it. It may therefore always be thought of as a $($free\/$)$
spectral variable}.
\end{itemize}
The converse is also true. Differentiation of the 3rd kind integrals
yields other integrals and functions. In other words, informally
speaking, one may say that

\begin{itemize}
\item \textit{Closed class of ODEs integrable through
$\Theta(x\,\boldsymbol{U}+\boldsymbol{D})$ is in fact integrable in
terms of $\{\Theta$, $\Theta'$, $\boldsymbol{\mathfrak{L}}_x\}$ and
`owes' to contain an external $($except for moduli\/$)$
\hbox{parameter $\p$}}.
\end{itemize}

Indeed, there is one fundamental logarithmic integral $\III$ for each
algebraic curve and it has a single parameter
$\p\Leftrightarrow\lambda$. In the elliptic case this $\III$ has form
\eqref{III} and for arbitrary genera it is expressed through the
$g$-dimensional $\Theta$'s \cite{baker}. The one fold logarithmic
$\partial_x$-derivative of \eqref{final}, \ie, derivatives of
$\III(\gamma\mathrm{'s};\p)$, yields the rational functions of
$(z,w)(\gamma\mathrm{'s})$ and therefore only meromorphic objects
remain since integrals themselves disappear. As a rough guide we have
here $\partial_x\exp(\III)=\III_x\!\cdot\exp(\III)$. This linear in
$\exp(\III)$ equation\footnote{An important remark is in order. We
might not say the same as applied to the `pure spectral' object
\eqref{dual} since non-indefinite integral remains. Again, Weierstrass'
theorem does the job. The `one-fold $\partial_x$', which is equivalent
here to the `$\partial_x$-differential closedness', explains
\textit{why the spectral equations are always linear.}} is treated as a
spectral one. Coefficients of this equation (more precisely its
$\lambda$-independent pieces) can be thought of as potential. The
quantity $\p$ should always be distinguished as an external one
because, otherwise, the following chain
$$
\boxed{\phantom{\big|}\textup{external\ } \lambda\quad
\hbox{\scalebox{1.3}[1]{$\Leftrightarrow$}}\quad
\partial_x\textup{-closedness}\quad
\hbox{\scalebox{1.3}[1]{$\Leftrightarrow$}}
\quad \textup{`$\Psi$\!-linearity'}\:\,}
$$
is destroyed altogether. Moreover, the $x$-dependence
$\Theta(x\,\boldsymbol{U}+\boldsymbol{D})$ is not bound to be a linear
one and spectral equations must not necessarily be of the form
\eqref{A}. Counterexamples in Sect.~\ref{9} illustrate these points.

\section{Definition of $\theta$ through Liouvillian extension\label{8}}
Insomuch as the object $\Lambda(x;\mathfrak{u}|\tau)$ is in fact a
theta-function with a parameter, we can use the differential equations
described above as the basis for a definition of the theta's themselves
and, in particular, consider character of their integrability.

\begin{proposition}\label{P3}
System \eqref{new} has the two algebraic $($rational\/$)$ integrals
$$
\vartheta_2^2\,\theta_4^2-
\vartheta_4^2\,\theta_2^2=A_1\,\vartheta_3^2\,\theta_1^2\,, \qquad
\vartheta_2^2\,\theta_3^2-
\vartheta_3^2\,\theta_2^2=A_2\,\vartheta_4^2\,\theta_1^2
$$
generalizing  the famous Jacobi $\theta$-identities when $A_1\ne 1\ne
A_2$.
\end{proposition}
\begin{proof}
The straightforward calculation shows that $\partial_xA_1=\partial_x
A_2=0$.
\end{proof}

In a nutshell, the differential genesis of the object $\Lambda$ is as
follows. Function $\Dtheta$ is determined differentially through
$\theta_1$. Therefore two functions with two arbitrary constants solve
the system \eqref{new}. One  of constants serves a homogeneity
$\theta\mapsto C\,\theta$ of \eqref{new}. Another one $\mathfrak{u}$ is
non-algebraic and is related to an autonomy  of equations \eqref{new}.
Hence the two transcendental functions $\theta_1(x-\mathfrak{u}|\tau)$
and $\theta_2(x-\mathfrak{u}|\tau)$ remain. These functions are
represented, up to a shift and holomorphic integral, by the one object
$\Lambda(x;\mathfrak{u}|\tau)$.

\begin{theorem}\label{T10}
Differential equations  \eqref{!}  are algebraically integrable.
\end{theorem}
\begin{proof}
By a direct computation one can show that any solution
$\theta_k=\theta$ of the system \eqref{!} satisfies the same 5th order
ODE
\begin{equation}\label{F}
\mbig[9](\frac{1}{F_x}\bigg(\frac{\pow{F}{x}{2}}{F}
\bigg)_{\!\!\!x}\mbig[9])_{\!\!\!x}
+8\,F_x=0\,,\quad F=(\ln\theta)_{\mathit{xx}}^{}
-2\,\kappa\,,\quad -\kappa\DEF 2\,\eta+\frac16\,\pi^2\,
(\vartheta_3^4+\vartheta_4^4)\,.
\end{equation}
Therefore, not taking into account $\mathfrak{u}$, there is only one
essential parameter in equations \eqref{!}, \ie, parameter $\kappa$.
From this it follows that
$$
F=\Xi(x;\boldsymbol{a},\boldsymbol{b},\boldsymbol{c}):\quad
\int\limits^{\;F}\!\!
\frac{dz}{\sqrt{z(z-\boldsymbol{a})(z-\boldsymbol{b})}}=
2\,\ri\,x+\boldsymbol{c}
$$
and integration is thus completed if, according to definition in
Remark~\ref{R4}, we  adjoin the inversion operation $\Xi$:
\begin{equation}\label{necanon}
\theta=\exp\!\!\int\limits^{\;x}\!\!\mbig[7]\{
\hbox{\small$\ds\int\limits^{\;x}$}
\Xi(y;\boldsymbol{a},\boldsymbol{b},\boldsymbol{c})\,dy
\mbig[7]\}\, dx \cdot
\re^{\kappa\,x^2+\boldsymbol{d}\,x+\boldsymbol{e}}_{}\,,
\end{equation}
where $\boldsymbol{a},\boldsymbol{b},\boldsymbol{c},
\boldsymbol{d},\boldsymbol{e}$ are the integration constants. Of
course, the inversion function $\Xi$ here bears no relation to ratios
of the $\theta$-series. Integration of equations for functions
$\Dtheta$ and $\Lambda$ is obvious.
\end{proof}

\begin{remark}\label{R8}
In a separate work we shall show that the algebraic integrability above
can be supplemented with a Hamiltonian formulation $\dot
X=\Omega\,\nabla\mathcal{H}(X)$ to the system \eqref{!} and its
Lagrangian description.
\end{remark}

It is noteworthy, that  variable $\theta$ satisfies an equation of
fifth order, not third, as it would be expected from the well-known
$\wp$-equation of Weierstrass \cite{ell}. Another point that should be
mentioned here is the fact that algebraic integrability of equations
\eqref{!} leads not merely to the $\theta$-function itself but  to the
elliptic $\boldsymbol{\mathfrak{L}}_x$-divisor  and even its
non-canonical extension by the quadratic exponent. (Notice that
constant $\kappa$ depends on modulus but constant $\boldsymbol{d}$ is
free.) Further, the two-fold integration of the \textit{transcendental}
inversion operation in \eqref{necanon} can be reduced to   integration
of an algebraic function---our base operation.

\begin{corollary}\label{C5}
The $\theta$-function can be defined through a meromorphic elliptic
integral.
\end{corollary}

To prove this it will suffice to make the following  substitution in
formula \eqref{necanon}:
$$
\int\limits^{\;x}\!
\Xi(y;\boldsymbol{a},\boldsymbol{b},\boldsymbol{c})\,dy=
\!\!\int\limits^{\;\Xi(x;\boldsymbol{a},
\boldsymbol{b},\boldsymbol{c})}
\!\!\!\!\!\!\!\!\!\!\!\!
\frac{z\,dz}{\sqrt{z(z-\boldsymbol{a})(z-\boldsymbol{b})}}\,.
$$

By this we obtain somewhat nonstandard way of introduction of the
$\theta$-functions. To all appearances, Tikhomandritski\u\i\ \cite{tih}
was the first to point out a way  of definition of the $\theta$ through
a meromorphic integral\footnote{He does not mention the quadratic
extension and differential closedness of the set
$\{\theta_k,\,\Dtheta\}$, however.} but his note \cite{tih} went
unnoticed in the literature. He poses a question about the natural
going from elliptic integrals to the theta-functions and presents the
mode of transition between these transcendents by introducing the
integral of the 2nd kind elliptic integral. Indeed, rewriting formula
\eqref{necanon} in the following form
\begin{equation}\label{exp}
\theta(x)=
\exp\!\! \int\limits^{\;x}\!\mbig[9]\{
\!\!
\hbox{\small$\ds\int\limits^{\;\Xi(x)}$}
\!\!
\Mfrac{z\,dz}{\sqrt{z(z-\boldsymbol{a})(z-\boldsymbol{b})}}
\!\mbig[9]\}\,dx\cdot
\re^{\kappa\,x^2+\boldsymbol{d}\,x+\boldsymbol{e}}_{}\,,
\end{equation}
we observe that such a way of introduction of a $\theta$-function
\textit{is in effect the result of Liouvillian extension of a
meromorphic integral}, \ie, adjoining an \textit{exponent of integral
of such an integral}. By this means we may adopt this point as a
\textit{differential definition} of the $\theta$  \`a la Liouville and,
subsequently, construct all the other objects of the theory:
meromorphic (algebraic) functions are the $\theta$-ratios, Abelian
integrals of 2nd kind are expressed through just introduced meromorphic
integral (or, which is the same, the $\theta'$), and the 3rd kind
integrals are the logarithmic ratios of the $\theta$'s with free
parameters (the $\Lambda$-objects). Holomorphic integrals are of course
the independent objects; they are not defined/determined through any
other ones. An important role of a meromorphic integral was already
observed by Clebsch \& Gordan in preface to their book
\cite[p.~VI]{clebsch} on the base of Jacobi's formula
$$
\frac{\mathit{\Theta}(u)}{\mathit{\Theta}(0)}=
\exp\!\int\limits_0^{\;u}\!\!Z(u)\,du\,,
$$
where $Z(u)$ is a Jacobi zeta-function notation in the theory of
elliptic functions  \cite{ell}.

The above differential properties of the $\Theta$'s splittable to
$\theta$'s raise the question as to whether the general
multi-dimensional $\Theta$-functions admit the similar `differential
kind'  definition. In particular, whether exists the purely Liouvillian
definition of an $x$-section $\Theta(x\,\boldsymbol{U}+D)$ like formula
\eqref{exp}\footnote{Roughly speaking, one needs an extensive
strengthening of Theorem~\ref{T7}; whence it follows that
$$
-2\ln_{\mathit{xx}}\!\!\Theta(x\,\boldsymbol{U}+\boldsymbol{D})=
\ln_{\mathit{xx}}\!\!\Psi+(\ln_x\!\!\Psi)^2+\mathrm{const}
$$
and, since the $\Psi$ is an exponent of the 3rd kind Abelian integral
(formula \eqref{final}), that is $\Psi=\exp\!\III(\gamma\text{'s})$,
the object $\Theta(x\,\boldsymbol{U}+\boldsymbol{D})$ itself is
computed as a `Liouvillian extension of Abelian integrals':
$$
\Theta(x\,\boldsymbol{U}+\boldsymbol{D})=\exp-\frac12\!\!
\left\{\III(\gamma\text{'s})+
\int\!\!\!\!\!\!\!\!\int\!\!\!\!\big[\III(\gamma\text{'s})\big]_x^2\,
dx\,dx \right\}\re^{a\,x^2+b\,x+c}
$$
which is a reminiscence of formula \eqref{necanon}.} or, if any, the
closed set of \textit{partial DEs} defining the complete set of the
general $\Theta,\,\Theta'(\boldsymbol{z})$-functions as ones of the $g$
arguments $\boldsymbol{z}$? Some relations between $\Theta$'s and
meromorphic Abelian integrals can be found in lectures by Weierstrass
(though no really this point has been mentioned in the modern
literature) but the question about \textit{closed and differentially
Liouvillian definition} (if it exists) of a $4^g$-set of the
$g$-dimensional $\Theta$-functions and associated derivatives $\Theta'$
remains an important open problem.

\section{Non-finite-gap integrable counterexamples\label{9}}
Definition of  integrability domain is not a subject of the
$\Theta$-function techniques. Therefore we may generate integrable
equations by any way differed from the classical FG-structure defined
by formula \eqref{theta} and Theorem~\ref{T5}. For example, equations
coming no from operators $\bL(\pot;\partial_x)$ by taking the canonical
spectral equation $\bL(\pot;\partial_x)\,\Psi=\lambda\,\Psi$ form in
general the operator $\lambda$-pencils
$\bL(\pot;\partial_x,\lambda)\,\Psi=0$, say, \eqref{NLS}. Their
solution structure is not known a priori\footnote{It is well known,
however, that Eq.~\eqref{NLS} is related to a matrix canonical
eigenvalue problem.}.  Moreover, we can even construct an equation
fitting no in the FG-scheme but having the same formal
$\Theta$-function form of solution.

\begin{example}
Omitting in notation the elliptic modulus $\tau$, elucidate the said
above by the following modification of the 2-gap Lam\'e equation:
$$
\Psi_\mathit{\!xx}=\big\{24\,\wp(2\,x)+8\,\wp(2\,x-\mathfrak{u})+
16\,\wp(\mathfrak{u})\big\}\,\Psi\,.
$$
It has a solution of the formal 2-gap Baker--Akhiezer form \eqref{6P}:
\begin{equation}\label{BA1}
\Psi(x;\mathfrak{u})= \frac{d}{dx}\frac{\Lambda(x;\mathfrak{u})}
{\theta_1(x)}\,,\qquad h=4\,\zeta(\mathfrak{u})-
2\,\zeta(2\,\mathfrak{u})
\end{equation}
(exercise: check this solution). By Theorems~\ref{T9} and \ref{T10}
this example is algebraically integrable over
$\mathbb{C}_\partial\big(\wp(2x),\wp(2x-\mathfrak{u})\big)$ (and over
$\mathbb{C}_\partial(\theta_1,\Lambda)$, of course) with solvable
Galois group but it has little in common with commutative
Burchnall--Chaundy operators, BA-function, or spectral lacunae. Formula
\eqref{BA1} shows that this $\Psi$-function  has no even a pole at
point $2\,x=\mathfrak{u}$ where  potential does. This pole depends on a
spectral parameter $\mathfrak{u}$ lying on an elliptic curve.
\end{example}

Nevertheless this example should not be considered as `too artificial'
because there exist the $\Theta$-function integrable models having
algebraic curves with $x$-de\-pen\-dent branch-points. Ernst's
equations in general relativity \cite{korotkin} provide a nice example
along these lines.

\subsection{Hermite's operator pencil\label{9.1}}
Consider now, in the framework of Picard--Vessiot theory, Hermitian
equation \eqref{hermite} itself. It is not a Burchnall--Chaundy
operator but the operator $\lambda$-pencil. By virtue of formulae
\eqref{hermite}--\eqref{pm}, we define this pencil over field $\Nu$ and
repeat  arguments about $\lambda$-dependence of  $R$.

One of solutions to this pencil is not exponential but a purely Abelian
meromorphic function \eqref{factor}, \ie, differential polynomial
$$
R_1=R(\dpotu;\lambda)\FED\boldsymbol{\mathrm{P}}\ni\Nu\,.
$$
Since base of solutions to Eq.~\eqref{hermite} is
$\{\pow{\Psi}{\!+}{2},\,\Psi_{\!+}\Psi_{\!-},\,\pow{\Psi}{\!-}{2}\}$
and $\pow{\Psi}{\!+}{2}\notin\Nu$, we may put the second solution as a
square of the BA-function $R_2=\pow{\Psi}{\!+}{2}$, where $\Psi_{\!+}$
is an adjoint transcendent
$$
\Psi_{\!+}\DEF\exp\!\!\!\int\limits^{\,\,x}\!\!
\frac{\boldsymbol{\mathrm{P}}_{\!\!x}+2\,\mu}
{2\,\boldsymbol{\mathrm{P}}}\,dx\,.
$$
The third solution is $R_3=\pow{\Psi}{\!-}{2}$ and we obtain that
\begin{equation}\label{prop}
R_1=\boldsymbol{\mathrm{P}}\,,\qquad R_2=\pow{\Psi}{\!+}{2}\,,\qquad
R_3=\frac{\boldsymbol{\mathrm{P}}^2}{\pow{\Psi}{\!+}{2}}\,.
\end{equation}
Hence rationality domain is the same as in Theorem~\ref{T1}, that is
$\Nu\langle\Psi_{\!+}\rangle$, and this extension is a Liouvillian one
of the transcendence degree 1 (we consider only the generic case
$\lambda\ne E_j$). We therefore can obtain the following result.

\begin{theorem}\label{T11}
Under the generic $\lambda\ne E_j$ the Galois group of Hermite's
equation \eqref{hermite} defined over $\Nu$ is connected and isomorphic
to the diagonal group $\G=\mathrm{Diag}(1,\alpha,\alpha^{\smin1})$.
Equations \eqref{hermite} is factorizable over field $\Nu$.
\end{theorem}
\begin{proof}
As  in proof of Theorem~\ref{T1} we perform a linear transformation of
the basis $\{R_1,R_2,R_3\}$ and check invariance of the base
differential relations between $R$'s. Let us take relations of the zero
and 1st order in derivatives:
$$
R_1=\boldsymbol{\mathrm{P}}\,,\qquad
R_2R_3=\boldsymbol{\mathrm{P}}^2\,,\qquad
(R_2)_x=\frac{\boldsymbol{\mathrm{P}}_{\!\!x}+2\,\mu}
{\boldsymbol{\mathrm{P}}}\,R_2\,,
\qquad
(R_3)_x=\frac{\boldsymbol{\mathrm{P}}_{\!\!x}-2\,\mu}
{\boldsymbol{\mathrm{P}}}\,R_3\,,
$$
which result from properties \eqref{prop}. Using  condition $\mu\ne 0$,
one easily derives that admissible transformations are
$$
R_1\mapsto 1     \cdot R_1\,,\qquad
R_2\mapsto \alpha\cdot R_2\,,\qquad
R_3\mapsto \delta\cdot R_3
$$
and $R_2R_3\mapsto 1\cdot R_2R_3$. Hence $\delta=\alpha^{\smin1}$.
Solutions $\{R_1,R_2,R_3\}$ are in general not algebraic functions,
hence $\alpha$ is a free nonzero complex number and we do not need
further to analyze the remaining relations of second order in
derivatives of $R$'s (they will be automatically satisfied). This
yields a connectivity of the group and the matrix\footnote{All this can
also be seen from the fact that Galois group belongs to
$\mathrm{SL}_3(\mathbb{C})$ and transcendence degree of the extension
is unity.} $\mathrm{Diag}(1,\alpha,\alpha^{\smin1})$. Factorizations of
Eq.~\eqref{hermite} are deducible by use of Liouville's scheme since we
know solutions to equation. For example
\begin{alignat*}{2}
\partial_{\mathit{xxx}}-4\,(u+\lambda)\,\partial_x-2\,u_x&=
\mbig[7](\partial_x+\frac{\boldsymbol{\mathrm{P}}_{\!\!x}+2\,\mu}
{\boldsymbol{\mathrm{P}}} \mbig[7])
\partial_x\!\!
\mbig[7](\partial_x-\frac{\boldsymbol{\mathrm{P}}_{\!\!x}+2\,\mu}
{\boldsymbol{\mathrm{P}}} \mbig[7])\\[0.3em]
&=\mbig[7](\partial_x+\frac{\boldsymbol{\mathrm{P}}_{\!\!x}+2\,\mu}
{\boldsymbol{\mathrm{P}}}\mbig[7])\!\!
\mbig[7](\partial_x-\frac{2\,\mu}{\boldsymbol{\mathrm{P}}}
\mbig[7])\!\!
\mbig[7](\partial_x-\frac{\boldsymbol{\mathrm{P}}_{\!\!x}}
{\boldsymbol{\mathrm{P}}} \mbig[7]).
\end{alignat*}
These factorizations are not unique because equation has order 3.
\end{proof}

\subsection{Inversion of non-holomorphic integrals}
As a last counterexample we consider an equation that leads, on the one
hand, to a nonstandard case of the inversion problem, and, on the other
hand, to a nonlinear $x$-evolution in a theta-function argument. As we
shall see, there is no essential difference between (quadrature)
integrability schemes of this example and those of pure FG-potentials.
This example was already pointed out as non-standard in \cite{ustinov}.

\begin{example}
Let us consider the following spectral problem
\begin{equation}\label{dym}
\Psi_{\mathit{\!xx}}=\frac{\lambda}{v^2}\,\Psi\,,
\end{equation}
where $v=v(x)$. As long as we have deal with invariant integration of
\eqref{dym} (Sect.~\ref{3}), the theory has just non-essential
modifications and we restrict ourselves to writing down all of its
attributes in a form of references source for the simplest but
nontrivial case $g=1$. We put
\begin{equation}\label{tr}
R(x;\lambda)=v\,\lambda-\phi(x)\,,\qquad
\phi(x)\DEF 2\,a(x-b)(x-c)\,,
\end{equation}
where $a,b,c\in\mathbb{C}$. Novikov's equation is the equation
$v^3v_{\mathit{xxx}}-4\,\phi\,v_x+4\,\phi_xv=0$ and its integrals are
as follows:
\begin{equation}\label{II}
\mu^2=\lambda^3+3\,E_2(\dpotv)\,\lambda^2+
E_1(\dpotv)\,\lambda+a^2(b-c)^2 \,,
\end{equation}
$$
3\,E_2=-\frac{1}{2}v\,v_\mathit{xx}+
\frac{1}{4}\,\pow{v}{x}{2}-2\,\frac{\phi }{v}\,,\qquad E_1=
\frac12\,\phi \,v_\mathit{xx}- \frac12\,\phi_x\,v_x+\frac{\phi
^2}{v^2}+ 2\,a\,v\,.
$$
'Trace formula' follows from a direct analogy of \eqref{factor}, \ie,
we set $R=(\lambda-\gamma)\,v$, but inversion problem becomes an
inversion of the logarithmic integral rather than Jacobi problem
\eqref{jac}. Indeed, manipulations with  integrals $E_1$, $E_2$ show
that
\begin{equation}\label{jj}
\int\limits^{\;\gamma}\!\!\frac{1}{z-E_2}
\frac{dz}{\sqrt{4\,z^3-g_2^{}z-g_3^{}}} =\frac{1}{2\,a(b-c)}\,\ln
\frac{x-b}{x-c}+D\,,
\end{equation}
where $g_2^{}=12\,\pow{E}{2}{2}-4\,E_1$ and
$g_3^{}=4\,E_1E_2-8\,\pow{E}{2}{3}-a^2(b-c)^2$.

From \eqref{jj} it follows that the $\theta$-function description
undergoes changes since $x$-evolution on jacobian is essentially
non-linear. We pass from parameter $E_2$ to $\varrho$ by the rule
$\wp(2\varrho)=E_2$ and represent the logarithmic integral \eqref{jj}
in terms of $\theta$-functions of the holomorphic one $\boldsymbol{r}$:
$\gamma=\wp(2\,\boldsymbol{r})$. We arrive at  a transcendental
equation determining function $\boldsymbol{r}=\boldsymbol{r}(x)$:
$$
\ln\frac{\theta_1(\boldsymbol{r}-\varrho)}
{\theta_1(\boldsymbol{r}+\varrho)}+
2\,\frac{\Dtheta(\varrho)}{\theta_1(\varrho)}\,\boldsymbol{r}=
\frac{\wpp(2\,\varrho)}{a\,(b-c)}\ln\frac{x-b}{x-c}+D\,.
$$
As a result we obtain that ultimate solution to the $\Psi$-function is
far from obvious:
$$
\Psi_{\mathit{\!xx}}=\frac{\lambda\,\wp^2(2\,\boldsymbol{r})}
{a^2(x-b)^2(x-c)^2}\,
\Psi\,,
\qquad
\Psi^\pm(x;\lambda)=\frac{\sqrt{(x-b)(x-c)}}
{\sqrt{\wp(2\,\boldsymbol{r})}}\cdot\!
\frac{\Lambda(\boldsymbol{r};\pm\mathfrak{u})}
{\theta_1(\boldsymbol{r})}\,,
$$
$$
\lambda=\wp(2\,\mathfrak{u})-\wp(2\,\varrho)\,,\qquad
h=2\,\eta\,\mathfrak{u}^2\,.
$$
\end{example}

A direct check of this solution is a good exercise in theta-calculus.
Moreover, nonlinearity of this $x$-evolution is two-fold. Logarithm
\eqref{jj} contains a fraction-linear function but the principal
nonlinearity comes from a transcendental nonlinearity of
$\boldsymbol{r}(x)$. It never becomes linear even though we replace the
general case in \eqref{tr}, that is $\phi(x)=2\,a(x-b)(x-c)$, with the
particular one $\phi(x)=\mathrm{const}$.

\section{Concluding remarks\label{10}}
The properties of $\theta$-functions outlined above differ in a crucial
respect from classical special functions since the latter ones are
defined by ODEs not integrable in quadratures over elementary or
algebraic functions. For example Bessel's functions or the Painlev\'e
transcendents. Therefore when generalizing rational theory (solitons),
not only do algebraic integrability takes place for Novikov's equations
but it also takes place  for linear spectral equations and even
$\theta$-functions. In all these cases the integration procedure has
been closed at a single and common step: adjunction of the inversion
operation $\Xi$. The elementary theory does not get by without
inversion as well:
$$
\int\!\!\text{rational functions}
\;\;\dashrightarrow\;\;\ln\;\;\dashrightarrow\;\;\text{inversion}\;\;
\dashrightarrow\;\;
\text{exponent}\;\;\dashrightarrow\;\;\text{solitons}\,.
$$

It should be also emphasized that it makes no difference whether 1st
kind integrals (Jacobi problem) or 2nd, 3rd kind ones have been
inverted (see, \eg, \eqref{jj}). The only thing is needed for the
(Liouvillian) algebraic integrability: inversion of  indefinite
integrals of any algebraic functions. Roughly speaking, the inversion
procedure appearing in `theta-methods' has also the Liouvillian
characterization because, according to important Eq.~\eqref{AB},
adjunction of any kind Abelian integrals above is allowed. Nontrivial
examples on inversion of meromorphic integrals can be found in
monograph \cite{GH}; they are associated with the Camassa--Holm
hierarchy and have also the theta-function description. In the same
place quite extensive bibliography is presented. In other words,
\textit{in regard to invariant integrability, the choice of the
$\Theta$-series or the inversion operation $\Xi$ is just a question of
nomenclature}. Introducing $\Theta$ is equivalent to removing
$\gamma$'s from formulae like \eqref{dubrovin}--\eqref{final} and
conversely. As for analytic representation of solutions, the
$\Theta$-series is of course the fundamental object.


\end{document}